\documentclass[11pt,a4paper]{article}
\usepackage[utf8]{inputenc}
\usepackage[T1]{fontenc}
\usepackage[english]{babel}
\usepackage{lmodern}
\usepackage{amsmath,amssymb,amsthm,mathtools,mathrsfs}
\usepackage[a4paper,left=2.05cm,right=2.05cm,top=1.9cm,bottom=2.1cm]{geometry}
\usepackage{microtype}
\microtypesetup{expansion=false}
\usepackage{booktabs,array,longtable}
\usepackage{enumitem}
\usepackage{xcolor}
\usepackage{graphicx}
\usepackage[hidelinks,bookmarks=true,bookmarksopen=true,pdfencoding=auto]{hyperref}
\usepackage{aliascnt}
\usepackage{needspace}
\usepackage{tikz}
\usetikzlibrary{arrows.meta,positioning,calc,fit}
\allowdisplaybreaks
\newtheorem{theorem}{Theorem}[section]
\newtheorem{introthm}{Theorem}

\newaliascnt{proposition}{theorem}
\newtheorem{proposition}[proposition]{Proposition}
\aliascntresetthe{proposition}
\newaliascnt{lemma}{theorem}
\newtheorem{lemma}[lemma]{Lemma}
\aliascntresetthe{lemma}
\newaliascnt{corollary}{theorem}
\newtheorem{corollary}[corollary]{Corollary}
\aliascntresetthe{corollary}
\theoremstyle{definition}
\newaliascnt{definition}{theorem}

\aliascntresetthe{definition}
\newaliascnt{remark}{theorem}
\newtheorem{remark}[remark]{Remark}
\aliascntresetthe{remark}
\newaliascnt{example}{theorem}

\aliascntresetthe{example}
\usepackage[nameinlink,noabbrev]{cleveref}
\crefname{introthm}{Theorem}{Theorems}\Crefname{introthm}{Theorem}{Theorems}
\crefname{theorem}{Theorem}{Theorems}\Crefname{theorem}{Theorem}{Theorems}
\crefname{proposition}{Proposition}{Propositions}\Crefname{proposition}{Proposition}{Propositions}
\crefname{lemma}{Lemma}{Lemmas}\Crefname{lemma}{Lemma}{Lemmas}
\crefname{corollary}{Corollary}{Corollaries}\Crefname{corollary}{Corollary}{Corollaries}
\crefname{definition}{Definition}{Definitions}\Crefname{definition}{Definition}{Definitions}
\crefname{remark}{Remark}{Remarks}\Crefname{remark}{Remark}{Remarks}
\crefname{example}{Example}{Examples}\Crefname{example}{Example}{Examples}

\newcommand{\RR}{\mathbb R}
\newcommand{\CC}{\mathbb C}
\newcommand{\HH}{\mathbb H}
\newcommand{\OO}{\mathbb O}
\newcommand{\Gtwo}{G_2}
\newcommand{\gtwo}{\mathfrak g_2}
\newcommand{\SO}{\mathrm{SO}}
\newcommand{\SU}{\mathrm{SU}}
\newcommand{\Sp}{\mathrm{Sp}}
\newcommand{\Aut}{\operatorname{Aut}}
\newcommand{\Der}{\operatorname{Der}}
\newcommand{\Stab}{\operatorname{Stab}}
\newcommand{\Lie}{\operatorname{Lie}}

\newcommand{\Ad}{\operatorname{Ad}}
\newcommand{\Span}{\operatorname{span}}

\hypersetup{pdftitle={Generating pairs in g2 related by an involution},pdfauthor={Santiago Pineda Montoya, Johan H. Rua Munoz, Borut Jurcic Zlobec},pdfkeywords={octonions, G2, Lie generation, quaternionic subalgebras, compact symmetric pairs, nongenerating locus}}

\title{\textbf{Generating pairs in $\mathfrak g_2$ related by an involution}}
\author{Santiago Pineda Montoya\thanks{Universidad Nacional de Colombia, Medell\'{i}n, Colombia. Corresponding author: \texttt{sapinedamo@unal.edu.co}.}
\and Johan H. R\'{u}a Mu\~{n}oz\thanks{Instituto de F\'{i}sica, Facultad de Ciencias Exactas y Naturales, Universidad de Antioquia, Medell\'{i}n, Colombia. Email: \texttt{heriberto.rua@udea.edu.co}.}
\and Borut Jur\v{c}i\v{c} Zlobec\thanks{Independent researcher, Ljubljana, Slovenia. Email: \texttt{borut.jurcic.zlobec@gmail.com}.}}
\date{September 2026\\[2pt]\small Expository revision: October 25, 2026 (referee amendments)}

\begin{document}
\maketitle
\begin{abstract}
Let $\mathbb O$ be the real octonion division algebra, let $\mathfrak g_2=\operatorname{Der}(\mathbb O)$, and fix a quaternionic subalgebra $H\subset\mathbb O$.
We ask when an element $Z$ preserving $H$ and its conjugate $TZT^{-1}$ generate $\mathfrak g_2$, where $T$ is a prescribed involutive octonion automorphism.
We prove that such a $Z$ exists if and only if $T(H)\ne H$.  Moreover, a fixed set of six derivations preserving $H$ contains a generating choice for every such $T$.
For an arbitrary $Z\in\mathfrak g_2$, we construct a polynomial criterion with two factors: one detects reducibility on the imaginary octonions, and the other detects generated algebras of dimension at most three.
An irreducible proper generated algebra must be a principal $\mathfrak{su}(2)$.  We characterize the elements admitting this case by their rotation frequencies and parametrize all corresponding involutions.
For one specified involution and quaternionic subalgebra $H$, we describe the nongenerating set by four geometric families and six irredundant systems of scalar equations, and prove that it has real dimension four.
\end{abstract}


\section{Introduction}
A generating pair in a Lie algebra gives two one-parameter subgroups with which to generate the corresponding connected group \cite{JurdjevicSussmann1972}.  When the available operations are one subgroup $e^{tZ}$ and one prescribed involution $T$, the second infinitesimal generator is forced to be $TZT^{-1}$.  This leads to a concrete question: can a fixed involution turn one one-parameter subgroup into a generating system, and which choices of $Z$ fail?  Bauer, Levaillant, and Freedman study the analogous universality question for a two-qubit gate used together with the swap of its two factors \cite{BauerLevaillantFreedman2014}.  Here we study the compact exceptional group $G_2$, with the further requirement that $Z$ preserve a prescribed quaternionic subalgebra of the octonions.  This keeps the continuous transformations within a chosen associative copy of the quaternions; the involution is the only available operation that can move that subalgebra.

Let $\OO$ be the real octonion division algebra and put
\[
G_2=\Aut(\OO),\qquad \gtwo=\Der(\OO),\qquad V=\operatorname{Im}\OO.
\]
For a quaternionic subalgebra $H\subset\OO$, write
\[
P_H=H\cap V,\qquad
\mathfrak k_H=\{Z\in\gtwo:Z(H)\subseteq H\}\cong\mathfrak{so}(4).
\]
Let $\mathscr C$ denote the nonidentity involutions in $G_2$, and set
$\mathfrak l_T(Z)=\Lie\{Z,TZT^{-1}\}$.
Every $T\in\mathscr C$ fixes a quaternionic subalgebra pointwise and negates its orthogonal complement.  Both the quaternionic subalgebras and these involutions are parametrized by $G_2/\SO(4)$.  Draper and Mart\'in-Gonz\'alez describe its geometry through associative planes and the maximal subalgebras of $\gtwo$ \cite{DraperMartin2025}.  Their description provides the invariant subspaces that occur in our nongenerating pairs.

\subsection{Generation from one stabilizer}
If $T(H)=H$, then both $Z$ and $TZT^{-1}$ lie in $\mathfrak k_H$ whenever $Z$ does, and cannot generate $\gtwo$.  Our first result proves that there is no other obstruction to existence.

\begin{introthm}\label{intro:A}\label{thm:universal-reflections}
For every quaternionic subalgebra $H\subset\OO$, there is a fixed set $\mathcal E_H\subset\mathfrak k_H$ of six explicitly specified derivations such that, for every $T\in\mathscr C$,
\begin{align*}
T(H)\ne H
&\quad\Longleftrightarrow\quad
\text{some }Z\in\mathfrak k_H\text{ satisfies }\mathfrak l_T(Z)=\gtwo\\
&\quad\Longleftrightarrow\quad
\text{some }Z\in\mathcal E_H\text{ satisfies }\mathfrak l_T(Z)=\gtwo.
\end{align*}
When these conditions hold, the generating elements form a nonempty real Zariski-open subset of $\mathfrak k_H$, with full Lebesgue measure.
\end{introthm}

The set $\mathcal E_H$ is independent of $T$; its six elements are listed in \eqref{eq:six-universal-pulses}.  The proof reduces the exceptional second quaternionic coordinates to an affine line.  Three choices on a sphere avoid that line, and two first-coordinate choices cover all involutions not preserving $H$.  An integral example in \Cref{thm:single} gives
\[
\left\langle\{e^{tX}:t\in\RR\},S\right\rangle=G_2
\]
for one explicitly fixed derivation $X$ and involution $S$.

Classical two-generation theorems choose the two elements freely \cite{Kuranishi1951,AlbuquerqueSilvaLeite1989,Bois2009}.  Explicit pairs for simple algebras and general genericity results are given in \cite{DetinkoDeGraaf2020,Chirvasitu2021}.  Theorem~A instead prescribes the conjugating involution and restricts the first element to a six-dimensional stabilizer.  It supplies a choice for every involution satisfying the geometric condition, not just for generic involutions.

\subsection{A test for an individual pair}
The existence question is different from deciding whether a specified pair generates.  For the latter, first allow $Z\in\gtwo$ without a stabilizer restriction.  Define
\[
A_\pm=\tfrac12(Z\pm TZT),\qquad C=[A_+,A_-],
\qquad C_\pm=[A_\pm,C].
\]
Then $\mathfrak l_T(Z)=\Lie\{A_+,A_-\}$.  The even and odd components are independent coordinates on the symmetric pair; the additional restriction in Theorem~A is $A_++A_-\in\mathfrak k_H$.
Use the Frobenius norm on matrices and let
\[
\mathscr E_T=\{Q\in\operatorname{End}(V):Q^{\mathsf T}=Q,\ \operatorname{tr}Q=0,\ QT=TQ\}.
\]
This space has dimension fifteen.  With any Frobenius-orthonormal basis $Q_1,\ldots,Q_{15}$, put
\[
\mathcal R(T,Z)=\det\bigl(\langle[Z,Q_i],[Z,Q_j]\rangle_F\bigr)_{i,j=1}^{15},
\]
\[
\mathcal H(T,Z)=\|A_-\|_F^2\|C_+\|_F^2+\|A_+\|_F^2\|C_-\|_F^2-2\|C\|_F^4.
\]
The determinant is independent of the chosen orthonormal basis.

\begin{introthm}\label{intro:B}\label{thm:universal-locus}
For every $T\in\mathscr C$ and $Z\in\gtwo$, the functions $\mathcal R$ and $\mathcal H$ are nonnegative, invariant under simultaneous $G_2$-conjugation, and polynomial in the matrix entries on $\mathscr C\times\gtwo$.  They satisfy
\[
\mathcal R(T,Z)=0\iff\mathfrak l_T(Z)\text{ is reducible on }V,
\qquad
\mathcal H(T,Z)=0\iff\dim\mathfrak l_T(Z)\le3.
\]
Consequently
\begin{equation}
\mathfrak l_T(Z)=\gtwo\quad\Longleftrightarrow\quad
\mathcal R(T,Z)\mathcal H(T,Z)>0.
\label{eq:universal-exact}
\end{equation}
If $\mathcal R(T,Z)>0$ and $\mathcal H(T,Z)=0$, the generated algebra is a principal $\mathfrak{su}(2)$.
\end{introthm}

Here principal means that the subalgebra acts irreducibly on the real seven-dimensional space $V$.  The factorization follows because the other two maximal proper types, $\mathfrak{su}(3)$ and $\mathfrak{so}(4)$, preserve a vector or an associative plane.  The factors have degrees thirty and eight in $Z$.  A polynomial equation for nongeneration can also be obtained formally from squared bracket minors; the factors above distinguish the invariant-subspace and three-dimensional obstructions.  The construction and its proof are in \Cref{sec:universal-locus}.

The principal case has a further geometric description.  A nonzero $Z$ admits a principal $\mathfrak l_T(Z)$ exactly when its positive rotation frequencies have ratio $1:2:3$.  For each such $Z$, all realizing involutions form a three-dimensional real-analytic manifold diffeomorphic to $(S^1)^2\times(0,\pi/2)$; see \Cref{cor:universal-spectral,prop:principal-reflection-fiber}.

\subsection{The exceptional set for a fixed involution}
We give a complete description for one relative position of $H$ and $T$.
Use the octonion convention \eqref{eq:fano}, let $H_A=\Span\{1,e_1,e_2,e_3\}$, and take $S$ to fix $e_1$, exchange $e_2,e_4$ and $e_3,e_5$, and negate $e_6,e_7$.
Identify $\OO=H_A\oplus H_Ae_4$.  Write a derivation in $\mathfrak k_A$ as
\[
Z(q_a,q_b)(x+ye_4)=[q_a,x]+(q_by-yq_a)e_4,
\]
\[
q_a=ae_1+q_Re_2+q_Ie_3,\qquad q_b=de_1+p_Re_2+p_Ie_3.
\]
Set $q=q_R+iq_I$, $p=p_R+ip_I$, $u=|q|^2$, $v=|p|^2$, and
\[
\psi=p(q^*)^3,\qquad
\delta=a^2+u-d^2-v,\qquad
\Theta=u\bigl((2v+3u)a-3ud\bigr)-(a+d)\operatorname{Re}\psi.
\]
Here $q^*$ denotes the complex conjugate of $q$; all components $a,d,q_R,q_I,p_R,p_I$ are real.

\begin{introthm}\label{intro:C}\label{thm:scalar-classifier}
For $Z\in\mathfrak k_A$, the algebra $\mathfrak l_S(Z)$ is proper if and only if at least one of the following systems holds:
\begin{equation}\label{eq:six-tests}\begin{aligned}
\textup{(i)}\quad &u=0;&
\textup{(ii)}\quad &a=0,\quad d^2+v=u;\\
\textup{(iii)}\quad &d=-a,\quad v=u;&
\textup{(iv)}\quad &a=d=0,\quad\operatorname{Im}\psi=0;\\
\textup{(v)}\quad &d=3a,\quad up=3q^3;&
\textup{(vi)}\quad &\delta=0,\quad\operatorname{Im}\psi=0,\quad\Theta=0.
\end{aligned}\end{equation}
Each system has a solution belonging to none of the other five.
The nongenerating set $\mathcal N\subset\mathfrak k_A\cong\RR^6$ is real algebraic of dimension four.  Any two points of its complement can be joined by two line segments in that complement.
\end{introthm}

The six systems need not be disjoint.  In particular, system (vi) also includes common-vector pairs: for $u>0$, the branch $d=a$, $up=q^3$ is one such case.  Irredundance means that each system has a point outside the union of the other five.

The proof identifies four geometric families: common-vector pairs, a projective family preserving associative planes, a resonant rank-six family, and a rank-four family without a common fixed vector.  Their parametrizations are given in \Cref{sec:failure}.  Elimination gives the six systems of \Cref{sec:scalar} and proves Theorem~C.  This provides a geometric description of the exceptions as well as equations for membership.

The structural input is the geometry of associative planes and their stabilizers \cite{HarveyLawson1982,Cacciatori2005,KnarrStroppel2025,ChemtovKarigiannis2022,DraperMartin2025}.  The explicit calculations use the standard octonion derivations \cite{Baez2002,SpringerVeldkamp2000,RauschSlupinski2022}.
After the preliminaries, \Cref{sec:varying-switch} proves Theorem~A, \Cref{sec:universal-locus} proves Theorem~B, and \Cref{sec:principal} treats the principal subalgebra case.  \Cref{sec:failure,sec:scalar} prove Theorem~C.

\emph{Version history.}  The generation criteria and classification also appear in our longer preprint \emph{Quaternionic Reflections and Lie Generation in $G_2$}, \href{https://arxiv.org/abs/2609.15497v1}{arXiv:2609.15497v1} (September 14, 2026).  The present version reorganizes that material and includes the coordinate explanations and Gram-block derivations within the manuscript; it is a revision of the same work.

\paragraph{Prerequisites and scope of self-containment.}
We assume finite-dimensional linear algebra, the basic correspondence between compact Lie groups and their Lie algebras, and elementary representation theory: invariant subspaces, complexification, and complete reducibility.  Quaternionic multiplication and the derivation rule are used throughout; the octonion multiplication convention is specified below.  We recall the particular $\SU(2)$ tensor identities and spectral conventions that enter the proofs.  Three kinds of structural input are used from the literature rather than reproved: the standard $G_2$ actions and stabilizers, the conjugacy classification of involutions, and the compact maximal-subalgebra list in \Cref{subsec:compact-maximals}.  The last is essential for the completeness assertions; it is not merely background motivation.  The final dimension argument also uses the standard dimension theorem for semialgebraic images, cited where it is applied.  Relative to these inputs, the proofs of the generation criteria and the explicit classification are contained in this article, including the block calculations in Appendix~\ref{app:gram-certificates}.  No companion document or computer program is needed to specify a proof step.

Theorem~B concerns unrestricted $Z\in\gtwo$, whereas Theorems~A and~C impose $Z\in\mathfrak k_H$ or $\mathfrak k_A$.  Keeping these domains distinct is important: the even and odd parts are independent in the unrestricted problem, but not after their sum is required to preserve a prescribed quaternionic subalgebra.  A reader interested first in the generation mechanisms may read \Cref{sec:setting,sec:varying-switch,sec:universal-locus}; \Cref{sec:principal} explains the exceptional irreducible outcome, and \Cref{sec:failure,sec:scalar} give the finer coordinate classification.

\section{Octonions, stabilizers, and involutions}
\label{sec:setting}
We collect the coordinates and representation-theoretic facts needed for the generation results.
The six-element construction in the next section uses the stabilizer model, its rotation frequencies, the compact maximal-subalgebra list, and the intersection bound proved below.

\subsection{Octonion conventions}
We use the compact real form $\gtwo=\Der(\OO)$ on $V=\operatorname{Im}\OO$.
Matrices act on columns, $[A,B]=AB-BA$, and $\Ad_gZ=gZg^{-1}$.
The trace inner product on skew operators is $\langle A,B\rangle=-\operatorname{tr}(AB)$.
The Euclidean product on $\OO$ is $\langle x,y\rangle=\operatorname{Re}(x\bar y)$, and $N(x)=x\bar x$.
In the orthonormal basis $1,e_1,\ldots,e_7$, multiplication is fixed by $e_i^2=-1$ and the oriented triples
\begin{equation}
(123),\ (145),\ (176),\ (246),\ (257),\ (347),\ (365).
\label{eq:fano}
\end{equation}
For each triple $(ijk)$, $e_ie_j=e_k$ and cyclically; reversing the factors changes the sign.
For $u,v\in V$, write $uv=-\langle u,v\rangle+u\times v$.
The associator $[x,y,z]=(xy)z-x(yz)$ is alternating.
A three-plane $P\subset V$ is associative when $\RR1\oplus P$ is quaternionic; it is then spanned by $u,v,u\times v$ for orthonormal $u,v$.
We write $P_H=H\cap V$ and take $H^\perp$ in $\OO$, so $P_H^\perp=H^\perp$ as subspaces of $V$.
The alternating form $\varphi(u,v,w)=\langle u\times v,w\rangle$ satisfies
$G_2=\{g\in\SO(V):g^*\varphi=\varphi\}$.
These conventions follow \cite{Baez2002,SpringerVeldkamp2000,ChemtovKarigiannis2022}.

For $a\in\OO$, let $L_a(x)=ax$ and $R_a(x)=xa$.  The standard inner derivations are
\[
D_{a,b}=[L_a,L_b]+[L_a,R_b]+[R_a,R_b],
\]
and satisfy
\[
D_{a,b}(x)=[[a,b],x]-3[a,b,x].
\]
We abbreviate $D_{ij}:=D_{e_i,e_j}$.  These maps obey the Leibniz rule, and the $D_{e_i,e_j}$ span the fourteen-dimensional algebra $\mathfrak g_2$ \cite{SpringerVeldkamp2000,RauschSlupinski2022}.  We use the ordered basis
\begin{equation}
\mathcal B_{G_2}=(D_{12},D_{13},D_{14},D_{15},D_{16},D_{17},D_{23},D_{24},D_{25},D_{26},D_{27},D_{45},D_{46},D_{47}).
\label{eq:g2basis}
\end{equation}

For $w\in V$ let $C_w\in\mathfrak{so}(V)$ be the cross-product operator
$C_w(x)=w\times x$.  The map $w\mapsto C_w$ is injective and its image is a
$\Gtwo$-submodule of $\mathfrak{so}(V)$ isomorphic to $V$; since $V$ and
$\gtwo$ are inequivalent $\gtwo$-modules, the classical splitting
\begin{equation}
\mathfrak{so}(V)=\gtwo\oplus\{C_w:w\in V\}
\label{eq:so7split}
\end{equation}
is orthogonal for the trace form $\langle A,B\rangle=-\operatorname{tr}(AB)$ \cite{ChemtovKarigiannis2022}.

\subsection{Quaternionic subalgebras and their stabilizers}\label{subsec:stabilizer-coordinates}
For a quaternionic subalgebra $H$, the plane $P_H=H\cap V$ is associative and $H=\RR1\oplus P_H$.  Conversely, every oriented associative three-plane $P\subset V$ determines the quaternionic subalgebra $\RR1\oplus P$.  The seven coordinate quaternionic subalgebras are
\[
H_{abc}=\Span\{1,e_a,e_b,e_c\}
\]
for the triples in \eqref{eq:fano}.  We write
\[
H_A=H_{123},\qquad H_B=H_{145},\qquad P_A=P_{H_A},\quad P_B=P_{H_B}.
\]

The transitive action of $G_2$ on associative three-planes and the identification of the stabilizer with $\SO(4)$ are classical \cite{HarveyLawson1982,Cacciatori2005,KnarrStroppel2025,ChemtovKarigiannis2022}.  Thus the space of quaternionic subalgebras is the homogeneous space $G_2/\SO(4)$.
We record the standard coordinate form of the stabilizer because it is used repeatedly below.

For any quaternionic subalgebra $H$, set
\[
K_H=\Stab_{G_2}(H)=\{g\in G_2:g(H)=H\},\qquad
\mathfrak k_H=\Lie(K_H)=\{Z\in\mathfrak g_2:Z(H)\subseteq H\}.
\]

For a coordinate quaternionic subalgebra we also write $\mathfrak k_{abc}=\mathfrak k_{H_{abc}}$.  We abbreviate $K_{H_A},\mathfrak k_{H_A}$ by $K_A,\mathfrak k_A$, and use the analogous convention for $H_B$ and for the pointwise stabilizers below.  Fix $H_0=H_A=H_{123}$ and identify it with $\HH$, so that $\OO=\HH\oplus\HH e_4$.  Write $\Sp(1)=\{a\in\HH:N(a)=1\}$ for the group of unit quaternions; its Lie algebra $\mathfrak{sp}(1)=\operatorname{Im}\HH$ uses the quaternionic commutator.  For $a,b\in\Sp(1)$ define
\begin{equation}
\gamma_{a,b}(x+ye_4)=axa^{-1}+(bya^{-1})e_4.
\label{eq:gamma}
\end{equation}

For an arbitrary quaternionic subalgebra $H$, write
\[
K_H^{\mathrm{pt}}=\{g\in K_H:g|_H=I\},\qquad
\mathfrak k_H^{\mathrm{pt}}=\Lie(K_H^{\mathrm{pt}})
=\{Z\in\mathfrak k_H:Z|_{P_H}=0\}.
\]

\begin{lemma}[Standard stabilizer model]
\label{prop:stabilizer}
The maps \eqref{eq:gamma} are octonion automorphisms preserving $H_0$.  The homomorphism $(a,b)\mapsto\gamma_{a,b}$ has kernel $\{(1,1),(-1,-1)\}$, and
\[
K_{H_0}=\Stab_{G_2}(H_0)
\cong(\Sp(1)\times\Sp(1))/\{\pm(1,1)\}
\cong\SO(4).
\]
The pointwise stabilizer of $H_0$ is $K_{H_0}^{\mathrm{pt}}\cong\Sp(1)$.  Formula \eqref{eq:gamma} is the usual $\SO(4)\subset G_2$ model; see, for example, \cite{Cacciatori2005}.
\end{lemma}

\begin{proof}
Using the Cayley--Dickson product
\[
(x+ye_4)(u+ve_4)=(xu-\bar v y)+(vx+y\bar u)e_4,
\]
associativity in $\HH$ gives $\gamma_{a,b}(zw)=\gamma_{a,b}(z)\gamma_{a,b}(w)$.  Its restriction to $H_0$ is $x\mapsto axa^{-1}$.  The only pairs acting trivially on both summands are $(1,1)$ and $(-1,-1)$.  Conversely, if $g\in G_2$ preserves $H_0$, then $g|_{H_0}$ is an inner quaternion automorphism $x\mapsto axa^{-1}$.  Orthogonality gives $g(e_4)=ue_4$ with $u\in\Sp(1)$, and multiplicativity forces
\[
g(ye_4)=(ua\,y\,a^{-1})e_4.
\]
Taking $b=ua$ gives \eqref{eq:gamma}.  Setting $a=1$, modulo the diagonal kernel, gives the pointwise stabilizer.
\end{proof}

For $H_0$, differentiating \eqref{eq:gamma} at $(1,1)$ gives convenient coordinates on $\mathfrak k_{H_0}$.  If $q_a,q_b\in\operatorname{Im}\HH$, the corresponding derivation acts by
\[
Z(q_a,q_b)(x+ye_4)=[q_a,x]+(q_by-yq_a)e_4.
\]
Thus every element of $\mathfrak k_{H_0}$ is represented uniquely by a pair $(q_a,q_b)\in\operatorname{Im}\HH\oplus\operatorname{Im}\HH$.  The same notation is transported to any quaternionic subalgebra by conjugation.  When it is written as two triples, their coordinates are taken in the ordered basis $(e_1,e_2,e_3)$ of $\operatorname{Im}H_A$.

For the two coordinate subalgebras $H_A,H_B$, the stabilizer bases are
\begin{align}
\mathfrak k_A&=\Span\{D_{12},D_{13},D_{23},D_{45},D_{46},D_{47}\},\\
\mathfrak k_B&=\Span\{D_{14},D_{15},D_{23},D_{26},D_{27},D_{45}\}.
\end{align}

For $Z=x_0D_{12}+x_1D_{13}+x_2D_{23}+x_3D_{45}+x_4D_{46}+x_5D_{47}$,
the quaternionic coordinates are obtained by comparing the Fano-basis action on $e_1,e_2,e_4$ with
\[
Z e_1=2a_3e_2-2a_2e_3,\qquad
Z e_2=-2a_3e_1+2a_1e_3,\qquad
Z e_4=(q_b-q_a)e_4.
\]
The first two evaluations determine $q_a$; the third then determines $q_b$.  This gives
\begin{equation}
 q_a=(a_1,a_2,a_3)=(2x_2-x_3,-2x_1-x_4,2x_0-x_5),\qquad
 q_b=(b_1,b_2,b_3)=3(x_3,x_4,x_5),
 \label{eq:terminal-qa-qb}
\end{equation}
The inverse of \eqref{eq:terminal-qa-qb} is
\[
(x_0,x_1,x_2,x_3,x_4,x_5)
=\left(\frac{a_3}{2}+\frac{b_3}{6},
-\frac{a_2}{2}-\frac{b_2}{6},\frac{a_1}{2}+\frac{b_1}{6},
\frac{b_1}{3},\frac{b_2}{3},\frac{b_3}{3}\right).
\]

Put $\alpha=\|q_a\|$ and $\beta=\|q_b\|$.
The \emph{rotation frequencies} of a skew operator on the seven-dimensional space $V$ are the three nonnegative numbers $\omega_j$, counted with multiplicity, for which its complex spectrum is $0,\pm i\omega_1,\pm i\omega_2,\pm i\omega_3$.  Each nonzero frequency contributes a real two-dimensional rotation block and hence two to the rank.
The action on $P_A$ is $2q_a\times(\cdot)$, and the commuting operators $L_{q_b}$ and $R_{q_a}$ act on $H_Ae_4$.
Consequently the rotation frequencies, including possible zeros, are
\[
2\alpha,\qquad |\beta-\alpha|,\qquad\beta+\alpha.
\]
This also holds in coordinates adapted to any quaternionic subalgebra.
The term \emph{principal spectrum} will mean the frequencies $\lambda,2\lambda,3\lambda$ for some $\lambda>0$.

\paragraph{Coordinate dictionary.}
The following notation will be used again in the fixed-involution classification.  All rows describe the same $Z\in\mathfrak k_A$; they do not introduce additional degrees of freedom.
\begin{center}
\renewcommand{\arraystretch}{1.22}
\begin{tabular}{@{}>{\raggedright\arraybackslash}p{2.6cm}>{\raggedright\arraybackslash}p{13.5cm}@{}}\toprule
Coordinates & Conversion and purpose\\\midrule
Derivation coefficients & $Z=x_0D_{12}+x_1D_{13}+x_2D_{23}+x_3D_{45}+x_4D_{46}+x_5D_{47}$.  This basis makes brackets explicit.\\
Quaternionic pair & $q_a=(2x_2-x_3,-2x_1-x_4,2x_0-x_5)$, $q_b=3(x_3,x_4,x_5)$.  These coordinates describe the action on $P_A\oplus H_A^\perp$.\\
Axial scalars & $a=a_1=2x_2-x_3$, $d=b_1=3x_3$.\\
Transverse scalars & $q=a_2+ia_3=(-2x_1-x_4)+i(2x_0-x_5)$,\newline $p=b_2+ib_3=3(x_4+ix_5)$.\\
Norms and phase & $u=|q|^2$, $v=|p|^2$, $\psi=p(q^\ast)^3$; thus $\alpha^2=a^2+u$, $\beta^2=d^2+v$.  The conjugation in $\psi$ is on $q$.\\
Rank and scalar tests & $\delta=\alpha^2-\beta^2$; $\rho$ and $\Theta$ are the explicit polynomials in \eqref{eq:scalar-theta}.  Together these quantities encode Theorem~C.\\\bottomrule
\end{tabular}
\end{center}
The scalar pair $(q,p)$ is complex, whereas $(q_a,q_b)$ is quaternionic.  In the local normal forms of \Cref{thm:low-complete,thm:rankfour-context-complete}, letters such as $b,c$ are temporary real coordinates, not the quaternionic subscripts $a,b$.

\subsection{The stabilizer representation}
Put $\mathfrak k=\mathfrak k_{H_0}$ and let $\mathfrak m=\mathfrak k^\perp\cap\mathfrak g_2$ for the trace form.  Write $V_m$ for the $(m+1)$-dimensional irreducible complex representation of $\Sp(1)\simeq\SU(2)$, and use $\boxtimes$ for the external tensor product of the two $\Sp(1)$ factors in the stabilizer.

For a real vector space or Lie algebra $E$, write $E_{\CC}=E\otimes_{\RR}\CC$.

Inner products extend complex-bilinearly in the orthogonal-algebra decompositions below.  The external tensor product means $(g_1,g_2)(v\otimes w)=g_1v\otimes g_2w$; $\Lambda^2E$ is the exterior square.  The isotropy representation is the action of $K_{H_0}$ on the tangent space of $G_2/K_{H_0}$ at the identity coset, identified here with $\mathfrak m$.

Here is a short way to check the tensor identities used below.  Realize $V_m$ as $\operatorname{Sym}^m(\CC^2)$; the diagonal element $\operatorname{diag}(t,t^{-1})$ has character
\[
\chi_m(t)=t^m+t^{m-2}+\cdots+t^{-m}.
\]
Multiplying the characters gives $\chi_2\chi_1=\chi_3+\chi_1$.  The weights in $\Lambda^2V_2$ are $2,0,-2$, so $\Lambda^2V_2\cong V_2$.  Finally, interchange of the two factors in $(E\otimes F)^{\otimes2}$ gives
\[
\Lambda^2(E\otimes F)\cong
(\operatorname{Sym}^2E\otimes\Lambda^2F)
\oplus(\Lambda^2E\otimes\operatorname{Sym}^2F).
\]
For the two standard two-dimensional modules, $\operatorname{Sym}^2V_1=V_2$ and $\Lambda^2V_1=V_0$.  These observations give all three identities in the next proof.  Complete reducibility justifies reading the decompositions from their characters; no further tensor-product formula is needed there.

\begin{lemma}[Isotropy branching and maximality]
\label{lem:maximality}
For $K=K_{H_0}\cong(\Sp(1)_a\times\Sp(1)_b)/\{\pm(1,1)\}$,
\begin{align}
V_{\CC}&\cong(V_2\boxtimes V_0)\oplus(V_1\boxtimes V_1),\label{eq:Vbranch}\\
\mathfrak g_{2,\CC}&\cong(V_2\boxtimes V_0)\oplus(V_0\boxtimes V_2)\oplus(V_3\boxtimes V_1).
\label{eq:gbranch}
\end{align}
Thus
\[
\mathfrak k_{\CC}\cong(V_2\boxtimes V_0)\oplus(V_0\boxtimes V_2),
\qquad
\mathfrak m_{\CC}\cong V_3\boxtimes V_1
\]
and $\mathfrak m$ is irreducible.  Therefore $\mathfrak k\cong\mathfrak{so}(4)$ is a maximal proper subalgebra of $\mathfrak g_2$, in agreement with the classical maximal-subalgebra classification \cite{Dynkin1952b}.
\end{lemma}

\begin{proof}
The action \eqref{eq:gamma} splits
\[
V=\operatorname{Im}H_0\oplus H_0e_4.
\]
The first summand is the adjoint three-dimensional module of $\Sp(1)_a$ and is trivial under $\Sp(1)_b$, giving $V_2\boxtimes V_0$.  On the second, $(a,b)$ acts on the quaternion coefficient by $y\mapsto bya^{-1}$, whose complexification is $V_1\boxtimes V_1$.  This proves \eqref{eq:Vbranch}.

Use $\mathfrak{so}(V_{\CC})\cong\Lambda^2V_{\CC}$ and
\[
\Lambda^2(A\oplus B)=\Lambda^2A\oplus(A\otimes B)\oplus\Lambda^2B.
\]
The tensor-product decomposition gives
\[
\Lambda^2(V_2)=V_2,
\quad
V_2\otimes V_1=V_3\oplus V_1,
\quad
\Lambda^2(V_1\boxtimes V_1)=(V_2\boxtimes V_0)\oplus(V_0\boxtimes V_2).
\]
Consequently
\[
\mathfrak{so}(V_{\CC})\cong
2(V_2\boxtimes V_0)\oplus(V_0\boxtimes V_2)
\oplus(V_3\boxtimes V_1)\oplus(V_1\boxtimes V_1).
\]
Removing the direct summand $V_{\CC}$ in \eqref{eq:so7split} yields \eqref{eq:gbranch}.  If $\mathfrak k\subseteq\mathfrak h\subseteq\mathfrak g_2$, the orthogonal splitting $\mathfrak g_2=\mathfrak k\oplus\mathfrak m$ gives $\mathfrak h=\mathfrak k\oplus W$ with $W$ a $\mathfrak k$-submodule of the irreducible module $\mathfrak m$.  Hence $W=0$ or $W=\mathfrak m$, proving maximality.
\end{proof}

\subsection{Involutions and parity}\label{subsec:parity-domain}
For a quaternionic subalgebra $H$, the map
\[
\sigma_H(h+y)=h-y\qquad(h\in H,\ y\in H^\perp)
\]
is $\gamma_{1,-1}$ in adapted coordinates and hence belongs to $G_2$.
Every nonidentity involution of the compact $G_2$ has this form \cite{DraperMartin2025}.
Thus
\[
\mathscr C=\{\sigma_H:H\subset\mathbb O\text{ quaternionic}\}
=\{T\in G_2:T^2=I,\ \operatorname{tr}_VT=-1\}.
\]
The positive and negative eigenspaces on $V$ have dimensions three and four.
Conjugation identifies $\mathscr C$ with the connected compact real algebraic manifold $G_2/\SO(4)$, of real dimension eight.
For $T\in\mathscr C$, write
\[
\mathfrak k_T=\{A\in\mathfrak g_2:TAT=A\},\qquad
\mathfrak m_T=\{B\in\mathfrak g_2:TBT=-B\}.
\]
The first space is the stabilizer algebra of the quaternionic subalgebra fixed by $T$; its dimension is six, and $\dim\mathfrak m_T=8$.
We will repeatedly use the following elementary observation.
For an involutive automorphism $\sigma$, set
\[
Z^+=\tfrac12(Z+\sigma Z),\qquad Z^-=\tfrac12(Z-\sigma Z).
\]
\begin{lemma}[Generation under an involution]
\label{lem:general-symmetric-pair}
Let $\mathfrak g$ be a finite-dimensional real Lie algebra and $\sigma$ an involutive automorphism.  For every $Z\in\mathfrak g$,
\[
\Lie\{Z,\sigma Z\}=\Lie\{Z^+,Z^-\},
\]
and this Lie algebra is $\sigma$-stable.
\end{lemma}

\begin{proof}
The pairs $(Z,\sigma Z)$ and $(Z^+,Z^-)$ have the same linear span, so they generate the same Lie algebra.  Since $\sigma Z^+=Z^+$ and $\sigma Z^-=-Z^-$, every homogeneous bracket is carried to itself up to its parity sign, proving stability.
\end{proof}
The inverse to $Z\mapsto(Z^+,Z^-)$ is addition.
In particular, for $\sigma=\operatorname{Ad}_T$ the two components vary independently in $\mathfrak k_T\oplus\mathfrak m_T$ when $Z$ is unrestricted.
\subsection{Proper subalgebras of the compact algebra}\label{subsec:compact-maximals}
For $0\ne v\in V$, put $\mathfrak g_2^v=\{Z\in\mathfrak g_2:Zv=0\}$.
Every proper subalgebra of the compact $\mathfrak g_2$ is contained in one of the following maximal subalgebras, up to $G_2$-conjugation \cite{DraperMartin2025,Dynkin1952b}:
\begin{center}
\begin{tabular}{@{}lll@{}}\toprule
Maximal subalgebra & Dimension & Action on $V$\\\midrule
$\mathfrak g_2^v\cong\mathfrak{su}(3)$ & $8$ & fixes $v$\\
$\mathfrak k_H\cong\mathfrak{so}(4)$ & $6$ & preserves $P_H\oplus H^\perp$\\
principal $\mathfrak{su}(2)$ & $3$ & irreducible\\\bottomrule
\end{tabular}
\end{center}
For the compact real form and these conjugacy classes, see Draper and Mart\'in-Gonz\'alez \cite{DraperMartin2025}, \S3.2, Proposition~3 and its proof in \href{https://arxiv.org/html/2504.07586v1}{arXiv:2504.07586v1}.  In that version, \S3.1 also states conjugacy of the compact principal subalgebras, and Remark~7 identifies their normalizer with $\SO(3)$ and their centralizer with the identity subgroup.
Here a principal $\mathfrak{su}(2)$ means the conjugacy class acting irreducibly on the seven-dimensional real space $V$.
Its complexification is the spin-three representation, so every nonzero element has positive rotation frequencies in the ratio $1:2:3$.
The other two maximal types act reducibly.
For a unit $v$, its group stabilizer is $\SU(3)$ on the complex three-space $v^\perp$, with complex structure $x\mapsto v\times x$.
These descriptions are the form of the classical classification needed below.
We use the exhaustiveness and conjugacy assertions of this classification as an external theorem.  Maximality of the quaternionic stabilizer has also been proved directly in \Cref{lem:maximality}; that argument alone would not establish that the three displayed types exhaust all maximal subalgebras.
The restriction of the positive invariant form makes every compact subalgebra reductive.  The three types above are the compact forms of the maximal reductive complex embeddings $A_2$, $A_1+A_1$, and the principal $A_1$ \cite{Dynkin1952b,DraperMartin2025}.

\subsection{The fixed involution used in the explicit calculations}
The general existence theorem allows any $T\in\mathscr C$.
For the later examples and classification we fix the signed permutation
\begin{equation}
\begin{aligned}
S(e_1)&=e_1,&S(e_2)&=e_4,&S(e_3)&=e_5,\\
S(e_4)&=e_2,&S(e_5)&=e_3,&S(e_6)&=-e_6,&S(e_7)&=-e_7.
\end{aligned}
\label{eq:switch}
\end{equation}
Its multiplicativity follows by checking the Fano products.  It satisfies $S^2=I$, $\operatorname{tr}_VS=-1$, and $S(H_A)=H_B$.

\section{Generation from a prescribed quaternionic stabilizer}\label{sec:varying-switch}
We prove the six-element assertion from the introduction.
The key estimate is that two different quaternionic stabilizers intersect in dimension at most three.
Its proof uses the two simple factors of the stabilizer and their actions on $V$.  The bound then restricts the nongenerating choices in $\mathfrak k_A$.

\subsection{Intersections of stabilizers}
\begin{lemma}[Intersections of distinct quaternionic stabilizers]
\label{lem:nonfixed-context-bound}
For any two distinct quaternionic subalgebras $H,H'$, one has
\[
\dim\bigl(\mathfrak k_H\cap\mathfrak k_{H'}\bigr)\le3.
\]
The bound is sharp.
\end{lemma}

\begin{proof}
Set $\mathfrak h=\mathfrak k_H\cap\mathfrak k_{H'}$ and identify $\mathfrak k_H=\mathfrak s_1\oplus\mathfrak s_2$ with $\mathfrak s_i\cong\mathfrak{su}(2)$.  Let $p_i$ be the two coordinate projections and put $\mathfrak n_1=\mathfrak h\cap(\mathfrak s_1\oplus0)$ and $\mathfrak n_2=\mathfrak h\cap(0\oplus\mathfrak s_2)$.  The spaces $p_i(\mathfrak h)$ are subalgebras of $\mathfrak{su}(2)$, so their dimensions are $0$, $1$, or $3$, and $\mathfrak n_i$ is an ideal in $p_i(\mathfrak h)$.  Moreover, pairing the two coordinates of elements of $\mathfrak h$ induces an isomorphism
\[
p_1(\mathfrak h)/\mathfrak n_1\simeq p_2(\mathfrak h)/\mathfrak n_2.
\]
If $\dim\mathfrak h\ge4$, this relation and the simplicity of $\mathfrak{su}(2)$ force one of the $\mathfrak n_i$ to be a full three-dimensional factor.  Indeed, if both projections are onto and both kernels vanish, $\mathfrak h$ is the graph of an isomorphism and has dimension three; all remaining possibilities of dimension at least four contain a full factor.  Thus $\mathfrak h$ contains one of the two simple factors of $\mathfrak k_H$.

The stabilizer model \eqref{eq:gamma} recovers $P_H$ from either factor.  The second factor fixes $P_H$ pointwise and acts irreducibly by quaternionic multiplication on $H^\perp$.  The first acts irreducibly on each of $P_H$ and $H^\perp$, of dimensions three and four.  In either case, $P_H$ is the unique invariant three-plane.  The full branching calculation is recorded in \Cref{lem:maximality}.  Since $P_{H'}$ is also invariant under the factor contained in $\mathfrak h$, it must equal $P_H$, forcing $H'=H$.  The contrapositive gives $\dim\mathfrak h\le3$.  For sharpness, take
\[
H=\exp(\pi D_{25}/12)H_A,
\]
whose associative frame may be taken as
\[
\frac{\sqrt3e_1+e_6}{2},\qquad
\frac{e_2+\sqrt3e_5}{2},\qquad
\frac{\sqrt3e_3-e_4}{2}.
\]
Direct substitution gives $S(H)\ne H$ and
\[
\mathfrak k_H\cap\mathfrak k_{S(H)}
=\Span\left\{
\tfrac12D_{17}+D_{24},\,-D_{15}+D_{26},\,D_{13}+D_{46}
\right\},
\]
which has dimension three.
\end{proof}
\subsection{Six elements suffice}
A subset of a real coordinate space is real algebraic when it is a common zero set of real polynomials; a real Zariski-open set is the complement of such a set.  The same terminology on a real algebraic manifold means the relative topology.  Full Lebesgue measure means that the complement has measure zero.
In the notation $Z(\mathbf a,\mathbf b)$ used next, the bold vectors are precisely the quaternionic coordinates $q_a,q_b$ from \eqref{eq:gamma}.
For $T\in\mathscr C$ and $\mathbf a\in P_A$ with $T\mathbf a\notin P_A$, define
\[
H_{\mathbf a,T}=\RR1\oplus\Span\{\mathbf a,T\mathbf a,\mathbf a\times T\mathbf a\}.
\]

\begin{lemma}[An affine bound for nongenerating choices]
\label{lem:universal-affine}
Let $T\in\mathscr C$ and let $0\ne \mathbf a\in P_A$ satisfy $T\mathbf a\notin P_A$.  The subspace $H_{\mathbf a,T}$ is a $T$-invariant quaternionic subalgebra distinct from $H_A$.  Suppose that $Z=Z(\mathbf a,\mathbf b)\in\mathfrak k_A$ has three pairwise distinct positive frequencies which are not in the ratio $1:2:3$.  Then
\[
\Lie\{Z,TZT\}\ne\mathfrak g_2
\quad\Longleftrightarrow\quad Z\in\mathfrak k_{H_{\mathbf a,T}}.
\]
For fixed $\mathbf a,T$, the set of $\mathbf b\in\operatorname{Im}\HH$ on the right is either empty or an affine subspace of dimension at most one.
\end{lemma}

\begin{proof}
The independent imaginary vectors $\mathbf a$ and $T\mathbf a$ generate the indicated quaternionic subalgebra by alternativity.  The involution exchanges them and sends their cross product to its negative, so the quaternionic subalgebra is $T$-invariant.  Since $T\mathbf a\notin P_A$, it is different from $H_A$.

The frequency hypothesis gives $\operatorname{rank}Z=6$ and $\ker Z=\RR \mathbf a$.  The conjugate has kernel $\RR T\mathbf a$, so the pair has no common fixed vector.  If its Lie algebra $L$ is proper, the compact maximal-subalgebra list recalled in \Cref{subsec:compact-maximals} leaves a vector stabilizer, a quaternionic stabilizer, or a principal $\mathfrak{su}(2)$.  The first is excluded by the common kernels; the last is excluded by the frequency ratios.  Thus $L\subset\mathfrak k_H$ for some quaternionic subalgebra $H$.

In this situation $\dim L>3$.  Indeed, an abelian $L$ would preserve the one-dimensional kernel of $Z$, and skew symmetry would force its elements to kill that line, contrary to the absence of a common fixed vector.  A nonabelian compact Lie algebra of dimension at most three is $\mathfrak{su}(2)$.  In $\mathfrak k_H\cong\mathfrak{su}(2)\oplus\mathfrak{su}(2)$, its projections onto the two factors are either zero or isomorphisms.  A factor embedding has either rank-four elements or positive frequencies in the ratio $1:1:2$; a graph embedding has equal quaternionic norms in the two factors and hence rank-four elements.  Indeed, in the coordinates $(q_a,q_b)$ both factors use the same quaternionic commutator, and a graph is the graph of an automorphism of $\mathfrak{su}(2)$.  Since $\Aut(\mathfrak{su}(2))\cong\SO(3)$, that automorphism preserves the quaternionic norm.  These assertions follow directly from the stabilizer frequencies $2\alpha,|\beta-\alpha|,\beta+\alpha$ in \Cref{subsec:stabilizer-coordinates}.  None is compatible with the assumed spectrum.

The generated algebra is $\Ad_T$-stable, so it lies in $\mathfrak k_H\cap\mathfrak k_{T(H)}$.  By \Cref{lem:nonfixed-context-bound}, dimension greater than three forces $T(H)=H$.  The skew restriction of $Z$ to the odd-dimensional plane $P_H$ has a kernel, so its unique kernel line $\RR \mathbf a$ lies in $P_H$.  Invariance gives $T\mathbf a\in P_H$ and therefore $H=H_{\mathbf a,T}$.  This proves necessity.  Conversely, $Z\in\mathfrak k_{H_{\mathbf a,T}}$ and $T(H_{\mathbf a,T})=H_{\mathbf a,T}$ place both generators in the same proper stabilizer.

If the affine set of $\mathbf b$ is nonempty, its direction space identifies with
\[
\mathfrak k_A^{\mathrm{pt}}\cap\mathfrak k_{H_{\mathbf a,T}}.
\]
This is a subalgebra of $\mathfrak k_A^{\mathrm{pt}}\cong\mathfrak{su}(2)$, so it has dimension zero, one, or three.  Dimension three would mean that the entire pointwise stabilizer preserves $P_{H_{\mathbf a,T}}$.  Its representation on $V$ is the direct sum of the trivial three-space $P_A$ and an irreducible real four-space.  Thus $P_A$ is its only invariant three-plane, forcing $H_{\mathbf a,T}=H_A$, a contradiction.  The dimension is therefore at most one.
\end{proof}

Choose the following vectors in the quaternionic coordinates:
\[
\mathbf a^{(1)}=e_1,\qquad \mathbf a^{(2)}=e_2,\qquad
\mathbf b^{(1)}=e_1+e_2,\quad \mathbf b^{(2)}=e_1+e_3,\quad \mathbf b^{(3)}=e_2+e_3,
\]
and define six derivations by
\[
U_{ij}=6Z(\mathbf a^{(i)},\mathbf b^{(j)}),\qquad i\in\{1,2\},\quad j\in\{1,2,3\}.
\]
Their coefficients in the ordered stabilizer basis $(D_{12},D_{13},D_{23},D_{45},D_{46},D_{47})$ are the rows
\begin{equation}
\begin{pmatrix}
0&-1&4&2&2&0\\
1&0&4&2&0&2\\
1&-1&3&0&2&2\\
0&-4&1&2&2&0\\
1&-3&1&2&0&2\\
1&-4&0&0&2&2
\end{pmatrix},
\label{eq:six-universal-pulses}
\end{equation}
in the order $(11,12,13,21,22,23)$.  

\Needspace{10\baselineskip}

\begin{proof}[Proof of Theorem~A]
If $T$ preserves $H_A$, then $T\mathfrak k_AT=\mathfrak k_A$, so no such pair can generate $\mathfrak g_2$.  Thus preservation of $H_A$ excludes both existence assertions.

Suppose $T(H_A)\ne H_A$.  At least one of $Te_1,Te_2$ is outside $P_A$: otherwise their cross product $Te_3=T(e_1\times e_2)$ also lies in $P_A$, forcing $T(P_A)=P_A$.  Fix such an $\mathbf a^{(i)}$.  Each $Z(\mathbf a^{(i)},\mathbf b^{(j)})$ has $\alpha=1$ and $\beta=\sqrt2$, hence positive frequencies
\[
2,\quad \sqrt2-1,\quad \sqrt2+1.
\]
They are pairwise distinct and not in the ratio $1:2:3$.  By \Cref{lem:universal-affine}, the failing $\mathbf b^{(j)}$ must lie in an affine subspace of dimension at most one.  The three distinct $\mathbf b^{(j)}$ lie on the sphere of radius $\sqrt2$; a line meets that sphere in at most two points.  Therefore at least one of them generates.  Scaling by six does not change Lie generation.  Inverting \eqref{eq:terminal-qa-qb} gives the integer coefficients in \eqref{eq:six-universal-pulses}.

To prove the assertion about the set of elements, fix $T$ and set
\[
E_1=\Span\{Z,TZT\},\qquad E_{r+1}=E_r+[Z,E_r]+[TZT,E_r].
\]
Inductively, $E_r$ is spanned by iterated brackets of length at most $r$, with
successive brackets taken with $Z$ or $TZT$.  These words span the generated
Lie algebra: the Jacobi identity rewrites an arbitrary bracket in this form.
If $E_{r+1}=E_r$, the space is invariant under both adjoint generators, so all
later words lie in $E_r$ and $E_r=\Lie\{Z,TZT\}$.  Otherwise its dimension
strictly increases.  Since $\dim\mathfrak g_2=14$, words of length at most
fourteen suffice for every element.
Their coordinate columns are polynomial in $Z$.  Full generation is thus the
nonvanishing of at least one of finitely many $14\times14$ minors, a real
Zariski-open condition.  The generating element proved above makes one such
minor a nonzero polynomial.  The nongenerating set is contained in its zero
set and therefore has Lebesgue measure zero.  Finally, conjugation transports
the stabilizer, the involution, and the six elements to any other
quaternionic subalgebra.
\end{proof}

A generating choice in the six-element set is, in particular, a generating choice in the full stabilizer.  The choice among the six elements depends on $T$.

\subsection{An explicit generating pair}\label{sec:single}
For the fixed $S$, Theorem~C gives a simple integral example.  We record it here to connect Lie generation with generation of the group by one one-parameter subgroup and $S$.
Fix the signed permutation $S$ in \eqref{eq:switch} and define
\begin{equation}
X=D_{46}+D_{47}\in\mathfrak k_A,
\qquad
Y=\Ad_SX=-(D_{26}+D_{27})\in\mathfrak k_B.
\label{eq:XY}
\end{equation}

\begin{proposition}[An integral generating pair]
\label{thm:single}
The pair in \eqref{eq:XY} satisfies
\[
\Lie\{X,Y\}=\mathfrak g_2.
\]
Consequently the subgroup generated by $\{e^{tX}:t\in\mathbb R\}$ and $S$ equals $G_2$.
\end{proposition}

\begin{proof}
Apply Theorem~C to $X$.  Its coordinates are
\[
a=d=0,\qquad q=-1-i,\qquad p=3+3i,
\]
so
\[
u=2,\qquad v=18,\qquad \psi=12i,\qquad
\delta=-16,\qquad up-3q^3=12i.
\]
None of the six systems in \eqref{eq:six-tests} holds.  Thus $X$ and $SXS$ generate $\gtwo$.
For clarity, the failures of systems (i)--(vi), respectively, are witnessed by
\[
u=2\ne0,\quad d^2+v=18\ne2=u,\quad v\ne u,
\quad\operatorname{Im}\psi=12\ne0,\quad up-3q^3=12i\ne0,
\quad\delta=-16\ne0.
\]

Finally, $Se^{tX}S^{-1}=e^{tY}$, so the subgroup $\Gamma_X$ generated by $S$ and $\{e^{tX}:t\in\mathbb R\}$ contains both one-parameter subgroups.
Let $E$ be the span of all $\operatorname{Ad}_gU$, with $g\in\Gamma_X$ and $U\in\{X,Y\}$.
It is invariant under $\operatorname{Ad}_{\Gamma_X}$.
Differentiating conjugation by $e^{tX}$ and $e^{tY}$ shows that $E$ contains every iterated bracket of $X,Y$; hence $E=\mathfrak g_2$.
Choose fourteen such conjugates forming a basis.  The product of their one-parameter subgroups lies in $\Gamma_X$ and has invertible differential at the origin.
The inverse-function theorem gives an identity neighborhood in $\Gamma_X$.
It follows that $\Gamma_X$ is open and therefore equals the connected group $G_2$.
This is the usual Lie-rank argument \cite{JurdjevicSussmann1972,SilvaLeiteCrouch1988,AgrachevSachkov2004}.
\end{proof}

\paragraph{A contrasting nongenerating pair.}
Take $Z_0=D_{13}$.  The same conversion gives
\[
(a,d,q,p)=(0,0,-2,0),\qquad (u,v,\psi,\delta,\Theta)=(4,0,0,4,0).
\]
It satisfies system (iv) of Theorem~C and none of the other five.  The geometry can also be seen directly: the Fano products give
$D_{13}e_2=0$, $D_{13}e_4=2e_6$, and $D_{13}e_6=-2e_4$.
Thus $Z_0$ preserves the associative plane $\Span\{e_2,e_4,e_6\}$, which is itself $S$-invariant.  Its conjugate preserves the same plane, so the pair lies in a proper quaternionic stabilizer.  Nevertheless $\ker Z_0=\RR e_2$ and $\ker(SZ_0S)=\RR e_4$: there is no common fixed vector.  This example explains why a common-kernel test alone cannot classify all failures.  The use of Theorem~C for the preceding generating example is a forward reference; the proof of Theorem~C does not use \Cref{thm:single}.

\section{A polynomial criterion for generation}\label{sec:universal-locus}
Fix $T\in\mathscr C$ and now allow $Z$ to range over all of $\mathfrak g_2$.
The maximal-subalgebra list in \Cref{subsec:compact-maximals} reduces nongeneration to two possibilities: a reducible action on $V$, or a principal three-dimensional subalgebra.
We detect the first by a symmetric commutant and the second by double commutators.
This approach gives both an equation for the exceptional set and the type of the remaining irreducible algebra.
The use of commutants in generation problems has precedents in \cite{ZeierSchulteHerbrueggen2011}.

Use the Frobenius inner product $\langle U,V\rangle_F=\operatorname{tr}(U^{\mathsf T}V)$ and the space $\mathscr E_T$ defined before Theorem~B.  Put
\[
\mathscr S_0=\{Q\in\operatorname{End}(V):Q^{\mathsf T}=Q,\ \operatorname{tr}Q=0\}.
\]
Then $\dim\mathscr S_0=27$ and $\dim\mathscr E_T=15$.  A representation is reducible if it preserves a nonzero proper real subspace.  Polynomial degrees below are in $Z$ for fixed $T$.

\subsection{A symmetric commutant detects reducibility}

A proper invariant subspace gives an orthogonal projector that commutes with
both generators.  The next lemma averages that projector under the involution
without losing its non-scalar part.  This replaces a search over invariant
subspaces by linear equations for a symmetric matrix.

\begin{lemma}[Odd-dimensional involution test]
\label{lem:odd-commutant}
Let $V$ be an odd-dimensional real Euclidean space, $T$ an orthogonal involution, and $Z\in\mathfrak{so}(V)$.  Then $\Lie\{Z,TZT\}$ acts reducibly on $V$ if and only if there is a nonzero traceless symmetric operator $Q$ such that
\[
[Q,T]=[Q,Z]=0.
\]
\end{lemma}

\begin{proof}
Write $\mathfrak l=\Lie\{Z,TZT\}$.  If such a $Q$ exists, it also commutes with $TZT$.  A nonzero traceless symmetric operator is not scalar, and any one of its proper eigenspaces is invariant under both generators.

Conversely, suppose $\mathfrak l$ is reducible.  Choose an invariant subspace $W$ with $1\le d:=\dim W\le(\dim V-1)/2$, replacing $W$ by its orthogonal complement if necessary.  Let $P_W$ be its orthogonal projector.  It commutes with both generators.  The algebra $\mathfrak l$ is stable under $\Ad_T$, so $TP_WT$ also commutes with both.  Set
\[
Q=P_W+TP_WT-\frac{2d}{\dim V}I.
\]
Then $Q$ is symmetric, traceless, and commutes with $T$ and $Z$.  It is nonzero: $P_W+TP_WT$ has positive trace $2d$ and rank at most $2d<\dim V$, and hence cannot be a scalar operator.  The rank inequality is the only point at which odd dimension is used.
\end{proof}

For $T\in\mathscr C$ and $Z\in\mathfrak g_2$, choose any Frobenius-orthonormal basis $Q_1,\ldots,Q_{15}$ of $\mathscr E_T$ and define
\begin{equation}
\mathcal R(T,Z)
=\det\bigl(\langle[Z,Q_i],[Z,Q_j]\rangle_F\bigr)_{i,j=1}^{15}.
\label{eq:R30-gram}
\end{equation}
This definition is independent of that choice of orthonormal basis.  With an arbitrary basis, divide the determinant on the right by $\det(\langle Q_i,Q_j\rangle_F)$.

For a formula not requiring an eigenspace basis, define operators on $\mathscr S_0$ by
\[
\Pi_T^{\pm}(Q)=\tfrac12(Q\pm TQT),\qquad
\mathsf A_Z(Q)=[Z,Q].
\]

\begin{proposition}[The reducibility determinant]
\label{prop:reducibility-polynomial}
The function $\mathcal R(T,Z)$ is nonnegative, is homogeneous of degree thirty in $Z$, and satisfies
\[
\mathcal R(T,Z)=0
\quad\Longleftrightarrow\quad
\Lie\{Z,TZT\}\text{ acts reducibly on }V.
\]
It is polynomial in the entries of $(T,Z)$ on $\mathscr C\times\mathfrak g_2$, without choosing an eigenspace frame.  The same function has the basis-independent expression
\begin{equation}
\mathcal R(T,Z)
=\det_{\mathscr S_0}
\bigl(\Pi_T^- -\Pi_T^+\mathsf A_Z^2\Pi_T^+\bigr).
\label{eq:R30-framefree}
\end{equation}
Both formulas are invariant under simultaneous orthogonal conjugation of the matrix data.
\end{proposition}

\begin{proof}
The matrix in \eqref{eq:R30-gram} is the Gram matrix of the linear map
\[
\mathscr E_T\longrightarrow\mathscr S_0,\qquad Q\longmapsto[Z,Q].
\]
Its determinant is nonnegative and vanishes exactly when the map has a nonzero kernel.  \Cref{lem:odd-commutant} identifies that kernel condition with reducibility.  Each Gram entry is quadratic in $Z$, giving degree thirty.

The operator $\mathsf A_Z$ preserves $\mathscr S_0$ and is skew-adjoint for the Frobenius metric.  The orthogonal splitting $\mathscr S_0=\mathscr E_T\oplus\operatorname{im}\Pi_T^-$ has dimensions $15+12$.  Relative to it, the operator in \eqref{eq:R30-framefree} is the Gram operator of $\mathsf A_Z|_{\mathscr E_T}$ on the first summand and the identity on the second.  Its determinant is therefore \eqref{eq:R30-gram}.  The entries of $\Pi_T^\pm$ and $\mathsf A_Z$ are polynomial in $T,Z$, establishing the basis-independent assertion.  Homogeneity of degree thirty is asserted on the involution variety, where these are complementary projections.  Conjugation is an isometry of all the spaces and maps involved.
\end{proof}

The fifteen-dimensional formula is convenient for evaluation; the twenty-seven-dimensional formula shows that the coefficients vary polynomially with the involution.

\subsection{A double-commutator criterion for dimension at most three}

The determinant $\mathcal R$ excludes reducible generated algebras, but an irreducible
proper generated algebra can still be a principal $\mathfrak{su}(2)$.  The compact maximal
list therefore leaves a question of dimension at most three.  That question
has an exact answer using only two double commutators.

For $T\in\mathscr C$ and $Z\in\mathfrak g_2$, put
\[
A_+=\tfrac12(Z+TZT),\qquad A_-=\tfrac12(Z-TZT),\qquad C=[A_+,A_-],
\]
\[
C_+=[A_+,C],\qquad C_-=[A_-,C].
\]
Define
\begin{equation}
\mathcal H(T,Z)
=\|A_-\|_F^2\|C_+\|_F^2
 +\|A_+\|_F^2\|C_-\|_F^2-2\|C\|_F^4.
\label{eq:H8}
\end{equation}

An invariant inner product on a real Lie algebra satisfies $\langle[X,Y],Z\rangle=\langle X,[Y,Z]\rangle$, equivalently every adjoint map is skew-adjoint.  An orthogonal involution preserves that inner product and has square equal to the identity.

More generally, if $\sigma$ is an orthogonal involution of a real Lie algebra
with a positive definite invariant inner product, use
$A_\pm=(Z\pm\sigma Z)/2$, $C=[A_+,A_-]$, and $C_\pm=[A_\pm,C]$ in the same
formula, and denote the resulting scalar by $\mathcal H(\sigma,Z)$.

\begin{lemma}[A criterion for dimension at most three]
\label{lem:H8}
Let a finite-dimensional real Lie algebra have a positive definite invariant inner product and an orthogonal involutive automorphism $\sigma$.  For every $Z$ in that algebra, the expression $\mathcal H(\sigma,Z)$ defined above is nonnegative and vanishes if and only if
\[
\dim\Lie\{Z,\sigma Z\}\le3.
\]
If $C=0$, the generated algebra is the abelian span of $A_+,A_-$.  If $C\ne0$ and $\mathcal H=0$, the generated algebra is exactly $\Span\{A_+,A_-,C\}\cong\mathfrak{su}(2)$.
\end{lemma}

\begin{proof}
The vectors $A_+,A_-$ are orthogonal by parity.  Invariance gives
\[
\langle C,A_+\rangle=\langle C,A_-\rangle=0,
\quad
\langle C_+,A_-\rangle=-\|C\|^2,
\quad
\langle C_-,A_+\rangle=\|C\|^2.
\]
Thus $\mathcal H$ is the sum of the two nonnegative Gram determinants
\[
\det\operatorname{Gram}(A_-,C_+)
+\det\operatorname{Gram}(A_+,C_-).
\]
If $C=0$, both determinants vanish and the generated algebra is abelian.  If $C\ne0$, then $A_+,A_-,C$ are nonzero and pairwise orthogonal.  Equality in the two Cauchy--Schwarz inequalities is equivalent to
\[
C_+=-\frac{\|C\|^2}{\|A_-\|^2}A_-,\qquad
C_-=\frac{\|C\|^2}{\|A_+\|^2}A_+.
\]
These relations close their span under brackets and give the compact three-dimensional simple algebra.  Conversely, if the generated algebra has dimension three and $C\ne0$, this orthogonal triple spans it.  Parity and invariance force the two displayed relations.  This proves both directions, including every case in which one of the vectors vanishes.
\end{proof}

\subsection{Combining the two factors}

The two tests now exhaust the compact maximal-subalgebra alternatives:
reducible generated algebras are detected by $\mathcal R$, and every remaining
proper generated algebra has dimension at most three.  Their product consequently gives
one equation for the nongenerating set, including degenerate elements.

For an involution $T\in\mathscr C$ and every $Z\in\mathfrak g_2$, set
\begin{equation}
\mathcal U(T,Z)=\mathcal R(T,Z)\mathcal H(T,Z).
\label{eq:U38}
\end{equation}

For a fixed quaternionic subalgebra $H_A$, also define the nongenerating set in the prescribed stabilizer
\[
\mathcal N_T=\{Z\in\mathfrak k_A:\Lie\{Z,TZT\}\ne\mathfrak g_2\}.
\]

\begin{proof}[Proof of Theorem~B]
Fix $T\in\mathscr C$ and $Z\in\mathfrak g_2$, and write
$L=\Lie\{Z,TZT\}$.  By \Cref{prop:reducibility-polynomial,lem:H8},
\[
\mathcal R=0\ \Longleftrightarrow\ L\text{ is reducible on }V,
\qquad
\mathcal H=0\ \Longleftrightarrow\ \dim L\le3.
\]
Both factors are nonnegative.  If $L=\mathfrak g_2$, its seven-dimensional
representation is irreducible and its dimension is fourteen; hence both
factors are positive.

Conversely, suppose $L$ is proper and choose a maximal proper compact
subalgebra containing it.  The list in \Cref{subsec:compact-maximals} gives
three possibilities.  A vector stabilizer preserves a line in $V$, and a
quaternionic stabilizer preserves the splitting $3\oplus4$; containment in
either makes $L$ reducible and forces $\mathcal R=0$.  A principal
$\mathfrak{su}(2)$ has dimension three, so containment in it forces
$\mathcal H=0$.  Thus every proper generated algebra makes the product vanish,
proving \eqref{eq:universal-exact} and its restriction
\eqref{eq:NT-exact}.

For the final assertion, suppose that
$\mathcal R>0$ and $\mathcal H=0$.  The generated algebra is irreducible.  It cannot
be contained in either reducible maximal type, so it lies in a principal
$\mathfrak{su}(2)$.  A proper subalgebra of $\mathfrak{su}(2)$ is abelian;
an abelian algebra of skew-symmetric operators in odd dimension has a common
fixed vector.  Thus $L$ equals that principal algebra.

By \Cref{prop:reducibility-polynomial} and the explicit bracket formula
\eqref{eq:H8}, the factors are polynomials on
$\mathscr C\times\mathfrak g_2$ of respective degrees thirty and eight in
$Z$.  Their product is homogeneous of degree thirty-eight.  Simultaneous
conjugation preserves the Frobenius metric, the brackets, and the spaces
$\mathscr E_T$, proving invariance.

\end{proof}

In particular,
\begin{equation}
\mathcal N_T=\{Z\in\mathfrak k_A:\mathcal U(T,Z)=0\}.
\label{eq:NT-exact}
\end{equation}
For the prescribed stabilizer, \Cref{thm:universal-reflections} also shows that $\mathcal U(T,\cdot)|_{\mathfrak k_A}$ is nonzero exactly when $T(H_A)\ne H_A$.
If $T(H_A)=H_A$, the operator $\Pi_{P_A}-3I/7$ is a nonzero traceless symmetric operator commuting with $T$ and every $Z\in\mathfrak k_A$, so $\mathcal R(T,\cdot)|_{\mathfrak k_A}=0$.
Here $\Pi_{P_A}$ is the orthogonal projector onto $P_A$.

\section{Principal three-dimensional generated algebras}\label{sec:principal}
The criterion of \Cref{sec:universal-locus} leaves one irreducible proper possibility.
We first exhibit it, then characterize the elements $Z$ that admit it and determine all the corresponding involutions for each fixed $Z$.

\subsection{An explicit principal subalgebra and involution}

The following basis of a principal subalgebra has its first element in $\mathfrak k_A$.

Define
\begin{align*}
K_{\rm p}&=\tfrac56D_{23}+\tfrac23D_{45}=Z(e_1,2e_1),\\
P_{\rm p}&=\tfrac{\sqrt6}{4}D_{14}
 +\tfrac{\sqrt{10}}6\bigl(\tfrac12D_{17}+D_{24}\bigr),\\
Q_{\rm p}&=\tfrac{\sqrt6}{4}D_{15}
 +\tfrac{\sqrt{10}}6\bigl(\tfrac12D_{16}-D_{25}\bigr).
\end{align*}
Put $N_{\rm p}=(K_{\rm p}+P_{\rm p})/\sqrt2$ and
\begin{equation}
T_*=I+\frac{136}{45}N_{\rm p}^2+\frac{10}{9}N_{\rm p}^4+\frac4{45}N_{\rm p}^6.
\label{eq:principal-reflection}
\end{equation}

For a triple with the cyclic commutator relations below, the quadratic Casimir in its action on $V$ is the operator $-(K_{\rm p}^2+P_{\rm p}^2+Q_{\rm p}^2)$.  On complexification, a spin-$j$ irreducible representation has dimension $2j+1$ and Casimir eigenvalue $j(j+1)$ in this normalization.  These conventions specify the irreducibility test in the proof.
In the notation of \Cref{lem:maximality}, the spin-$j$ module is $V_{2j}$, with $j$ a nonnegative half-integer.  A normalized rotation generator has eigenvalues $im$ for $m=-j,-j+1,\ldots,j$.  Thus spin three has dimension seven and rotation frequencies $1,2,3$.  The Casimir commutes with the three generators and acts by the stated scalar on each irreducible summand.  Consequently the identity $-(K^2+P^2+Q^2)=12I$ forces $j(j+1)=12$, hence $j=3$, on every summand.  On a seven-dimensional space there is room for exactly one such summand.  This explains the irreducibility inference without requiring a choice of highest-weight basis in $V$.

\begin{proposition}[An explicit principal generating pair]
\label{prop:principal-moving}
These derivations satisfy
\begin{equation}
[K_{\rm p},P_{\rm p}]=Q_{\rm p},\quad [P_{\rm p},Q_{\rm p}]=K_{\rm p},\quad [Q_{\rm p},K_{\rm p}]=P_{\rm p},
\qquad -(K_{\rm p}^2+P_{\rm p}^2+Q_{\rm p}^2)=12I_V.
\label{eq:principal-triple}
\end{equation}
Then $T_*\in\mathscr C$, $T_*(H_A)\ne H_A$, and
\[
T_*K_{\rm p}T_*=P_{\rm p},\qquad
\Lie\{K_{\rm p},T_*K_{\rm p}T_*\}=\Span\{K_{\rm p},P_{\rm p},Q_{\rm p}\}
\]
is a principal $\mathfrak{su}(2)$.  In particular,
\[
\mathcal R(T_*,K_{\rm p})>0,\qquad \mathcal H(T_*,K_{\rm p})=0.
\]
\end{proposition}

\begin{proof}
The identities in \eqref{eq:principal-triple} follow by substituting the fixed Fano derivations.  They give a compact $\mathfrak{su}(2)$ representation whose Casimir is $12I$.  In the normalization of the displayed brackets, the Casimir on a spin-$j$ irreducible module is $j(j+1)$.  Thus only spin three occurs; since $\dim V=7$, the representation is irreducible, and the subalgebra is principal.  Explicit principal triples in compact $\mathfrak g_2$ are classical; see, for example, Draper and Mart\'in-Gonz\'alez \cite{DraperMartin2025}. The displayed normalization is chosen to put $K_{\rm p}$ in the specified stabilizer.

The element $N_{\rm p}$ is conjugate to $K_{\rm p}$ inside this compact subgroup, so its spectrum is $0,\pm i,\pm2i,\pm3i$.  The polynomial on the right of \eqref{eq:principal-reflection} takes values $1,-1,1,-1$ on the eigenspaces of $N_{\rm p}^2$ with eigenvalues $0,-1,-4,-9$.  Consequently it equals $\exp(\pi N_{\rm p})$, is an involution in $G_2$, and has trace $-1$ on $V$.  In the adjoint three-dimensional algebra it is the half-turn about the axis $K_{\rm p}+P_{\rm p}$, hence sends $K_{\rm p}$ to $P_{\rm p}$.  The latter does not belong to $\mathfrak k_A$, so $T_*$ cannot preserve $H_A$.  The two generators span the indicated principal algebra together with their bracket.  \Cref{prop:reducibility-polynomial} and \Cref{lem:H8} give the final two assertions.
\end{proof}

The example belongs to a continuous family.  For
\[
T_\theta=\exp\bigl(\pi(\cos\theta K_{\rm p}+\sin\theta P_{\rm p})\bigr),
\]
all $T_\theta$ are involutions and
\[
T_\theta K_{\rm p}T_\theta=\cos(2\theta)K_{\rm p}+\sin(2\theta)P_{\rm p}.
\]
Whenever $\sin(2\theta)\ne0$, one has $T_\theta(H_A)\ne H_A$ and $\Lie\{K_{\rm p},T_\theta K_{\rm p}T_\theta\}$ is principal.  

\subsection{Characterization by rotation frequencies}
The characterization uses the vector stabilizer $\SU(3)$ and its standard action on the orthogonal complex three-space \cite{DraperMartin2025,Baez2002}.

\begin{corollary}[Spectral realizability of the principal obstruction]\label{cor:universal-spectral}
For a nonzero $Z\in\mathfrak g_2$, the following conditions are equivalent:
\begin{enumerate}[label=\textnormal{(\roman*)},leftmargin=9mm]
\item the positive rotation frequencies of $Z$ have ratio $1:2:3$;
\item $Z$ belongs to a principal $\mathfrak{su}(2)$ subalgebra of $\mathfrak g_2$;
\item there is $T\in\mathscr C$ such that $\Lie\{Z,TZT\}$ is a principal $\mathfrak{su}(2)$.
\end{enumerate}
For $Z=Z(q_a,q_b)\in\mathfrak k_A$ with $\alpha>0$, these conditions are equivalent to
\[
\frac{\beta}{\alpha}\in\left\{2,5,\frac13\right\}.
\]
Every involution realizing (iii) for such a $Z\in\mathfrak k_A$ satisfies $T(H_A)\ne H_A$.
Outside the principal spectrum, including $Z=0$, generation for any $T\in\mathscr C$ is equivalent to $\mathcal R(T,Z)>0$.
\end{corollary}
\begin{proof}
A nonzero element of the irreducible spin-three representation has positive frequencies in the ratio $1:2:3$, so (iii) implies (ii) and (ii) implies (i).
To prove the converse, choose a unit vector $v\in\ker Z$.  Under (i) the kernel is a line.  The vector stabilizer $\mathfrak g_2^v\cong\mathfrak{su}(3)$ acts on $v^\perp$ as a complex three-space with complex structure $x\mapsto v\times x$.
The skew-Hermitian operator $Z|_{v^\perp}$ has eigenvalues $it_1,it_2,it_3$, where
\[
t_1+t_2+t_3=0,\qquad \{|t_1|,|t_2|,|t_3|\}=\{\lambda,2\lambda,3\lambda\}.
\]
The only signed triples with these absolute values and zero sum are permutations of $\lambda(1,2,-3)$ and its negative.
Conjugate $v$ to $e_1$ using $G_2$ and diagonalize in its $\SU(3)$ stabilizer.  Unitary diagonalization may be taken in $\SU(3)$ by changing the scalar phase of the diagonalizing matrix.
The element $K_{\rm p}=Z(e_1,2e_1)$ has the signed triple $(2,1,-3)$ in this complex structure.
It follows that $Z=c\Ad_g K_{\rm p}$ for some $g\in G_2$ and $c\ne0$.
Taking $T=gT_*g^{-1}$ and using \Cref{prop:principal-moving} gives
\[
TZT=c\Ad_g P_{\rm p},\qquad
\Lie\{Z,TZT\}=\Ad_g\Span\{K_{\rm p},P_{\rm p},Q_{\rm p}\},
\]
which proves (iii).  

For $Z\in\mathfrak k_A$ with $\alpha>0$, divide the frequencies $2\alpha,|\beta-\alpha|,\beta+\alpha$ by $\alpha$ and put $r=\beta/\alpha$.
For $r>1$ their signed complex triple is $(2,r-1,-r-1)$; the ratio $1:2:3$ gives $r=2$ or $r=5$.
For $0\le r<1$, the largest frequency is $2$, and $1-r$, $1+r$ must be $2/3$, $4/3$, giving $r=1/3$.
The cases $r=1$ or $\alpha=0$ have a zero frequency and are not principal.
If a realizing $T$ preserved $H_A$, its generated principal algebra would lie in $\mathfrak k_A$, which preserves the proper subspace $P_A$.  This contradicts irreducibility.
Finally, without a principal spectrum the irreducible proper case of \Cref{thm:universal-locus} is impossible, leaving exactly the asserted reducibility test.
\end{proof}

For $Z\in\mathfrak k_A$, the elements satisfying the preceding corollary form the nonzero part of
\[
(\beta^2-4\alpha^2)(\beta^2-25\alpha^2)(9\beta^2-\alpha^2)=0.
\]

\subsection{The involutions for a fixed element}
For a vector space $E$ over $\mathbb R$ or $\mathbb C$, $\mathbb RP(E)$ or $\mathbb CP(E)$ denotes its space of lines over the respective field.
We use $\mathbb RP^m=\mathbb RP(\mathbb R^{m+1})$ and similarly for $\mathbb CP^m$.

The preceding corollary describes the projection onto the $Z$-coordinate.  We now fix a nonzero element with principal spectrum and determine every realizing involution.
Put
\[
\mathfrak s_{\rm p}=\Span\{K_{\rm p},P_{\rm p},Q_{\rm p}\},\qquad
C_{\rm p}=\{u\in G_2:\Ad_uK_{\rm p}=K_{\rm p}\},
\]
and let $S_{\rm p}\subset G_2$ be the connected subgroup with Lie algebra $\mathfrak s_{\rm p}$.
Its spin-three action identifies it with $\SO(3)$.
For a Lie subalgebra $\mathfrak s$, its normalizer is $N_{G_2}(\mathfrak s)=\{u\in G_2:\Ad_u\mathfrak s=\mathfrak s\}$.
For any fixed $Z$, define
\[
\mathscr P_Z=\{T\in\mathscr C:\Lie\{Z,TZT\}\text{ is principal}\}.
\]
For $Z$ with principal spectrum choose, once and for all, $c\ne0$ and $g\in G_2$ with $Z=c\Ad_gK_{\rm p}$, as in \Cref{cor:universal-spectral}.
The family $T_\theta$ is the one defined after \Cref{prop:principal-moving}.

\begin{proposition}[Involutions giving a principal generated algebra]\label{prop:principal-reflection-fiber}
The centralizer $C_{\rm p}$ is a two-dimensional torus.  For the fixed choice $Z=c\Ad_gK_{\rm p}$,
\[
\mathscr P_Z
 =\{g uT_\theta u^{-1}g^{-1}:u\in C_{\rm p},\quad 0<\theta<\pi/2\}.
\]
Every involution on the right has a unique pair $(u,\theta)$ for the fixed choice of $c,g$.  This parametrization is a real-analytic diffeomorphism onto $\mathscr P_Z$ with the subspace topology inherited from $\mathscr C$.  The fiber is therefore an embedded real-analytic submanifold of dimension three and codimension five in the eight-dimensional involution manifold $\mathscr C$.
If $Z\in\mathfrak k_A$, all these involutions move $H_A$.
\end{proposition}
\begin{proof}
Scaling and conjugation reduce the proof to $Z=K_{\rm p}$.
First compute its group centralizer.  The kernel line is $\RR e_1$.  An element centralizing $K_{\rm p}$ must preserve this line.
If it sent $e_1$ to $-e_1$, its restriction to $e_1^\perp$ would be conjugate-linear for $C_{e_1}$, since it preserves the cross product.
Commutation with $K_{\rm p}$ would then identify the signed complex spectrum $(2,1,-3)$ with its negative, which is impossible.
It therefore fixes $e_1$ and lies in its $\SU(3)$ stabilizer.
There $K_{\rm p}$ has three distinct complex eigenvalues, so its centralizer is the diagonal determinant-one unitary torus.  Hence $C_{\rm p}\cong(S^1)^2$.

The group centralizer of $\mathfrak s_{\rm p}$ in $G_2$ is trivial.  Indeed, its real seven-dimensional action is irreducible.
A commuting symmetric operator is scalar, and a commuting skew-symmetric operator in odd dimension has a nonzero kernel and must vanish by irreducibility.
Thus the real commutant consists of scalars; its orthogonal elements are $\pm I_7$, of which only $I_7$ lies in $\SO(7)$.
Every automorphism of $\mathfrak{su}(2)$ is inner and is realized by $S_{\rm p}$, so
\[
N_{G_2}(\mathfrak s_{\rm p})=S_{\rm p}.
\]
Every principal algebra containing $K_{\rm p}$ is a $C_{\rm p}$-conjugate of $\mathfrak s_{\rm p}$.
To see this, choose $h$ carrying $\mathfrak s_{\rm p}$ to such an algebra, using the compact conjugacy in \Cref{subsec:compact-maximals}.
Then $\Ad_{h^{-1}}K_{\rm p}\in\mathfrak s_{\rm p}$ has the same trace norm as $K_{\rm p}$.
Rotations in $S_{\rm p}$ are transitive on that sphere, so $\Ad_vK_{\rm p}=\Ad_{h^{-1}}K_{\rm p}$ for some $v\in S_{\rm p}$.
The element $hv$ centralizes $K_{\rm p}$ and carries $\mathfrak s_{\rm p}$ to the prescribed algebra.

If $T\in\mathscr P_{K_{\rm p}}$, it normalizes its generated principal algebra: conjugation by $T$ interchanges the two generators.
Choose $u\in C_{\rm p}$ with $\Lie\{K_{\rm p},TK_{\rm p}T\}=\Ad_u\mathfrak s_{\rm p}$, as in the preceding paragraph.
Then $u^{-1}Tu\in N_{G_2}(\mathfrak s_{\rm p})=S_{\rm p}$ is a nonidentity involution.
It is therefore a half-turn about an unoriented axis in $S_{\rm p}\cong\SO(3)$.  Write its unit axis in the normalization $(K_{\rm p},P_{\rm p},Q_{\rm p})$ as
$\xi_1K_{\rm p}+\xi_2P_{\rm p}+\xi_3Q_{\rm p}$, where $\sum\xi_j^2=1$.
Its conjugate of $K_{\rm p}$ is independent of $K_{\rm p}$ exactly when $0<|\xi_1|<1$.
Choose the sign of the axis so that $\xi_1>0$ and rotate its transverse direction by $\exp(tK_{\rm p})\in C_{\rm p}$.
The axis becomes $\cos\theta K_{\rm p}+\sin\theta P_{\rm p}$ with $0<\theta<\pi/2$.
Conversely, every such half-turn generates $\mathfrak s_{\rm p}$ with $K_{\rm p}$ and is an involution, by the spin-three weights.
This proves the asserted equality of sets.

The angle is recovered from
\[
\frac{\langle K_{\rm p},TK_{\rm p}T\rangle_{\rm tr}}{\|K_{\rm p}\|_{\rm tr}^2}=\cos(2\theta).
\]
If two centralizer elements give the same involution at this angle, their quotient commutes with $K_{\rm p}$ and $T_\theta$.
It therefore commutes with $K_{\rm p}$ and $T_\theta K_{\rm p}T_\theta$, hence with all of $\mathfrak s_{\rm p}$, and is the identity.
The parametrization is thus injective.  Its torus-orbit differential is injective by the same stabilizer argument; the angle has nonzero derivative in the displayed scalar on the open interval.
It is an analytic immersion.  On every compact subinterval the domain is compact, and the displayed scalar recovers $\theta$ continuously in the subspace topology of $\mathscr C$.  It follows that the map is an embedding with analytic inverse onto its image.  Its codimension is $\dim\mathscr C-3=8-3=5$.
The last assertion follows from \Cref{cor:universal-spectral}.
\end{proof}

For fixed $c$, changing $g$ to $gv$ with $v\in C_{\rm p}$ replaces the torus coordinate $u$ by $v^{-1}u$.

\section{The nongenerating set for the fixed involution \texorpdfstring{$S$}{S}}\label{sec:failure}
We return to $Z\in\mathfrak k_A$ and the signed permutation $S$ in \eqref{eq:switch}.
Write
\[
L_Z=\Lie\{Z,SZS\},\qquad
\mathcal N=\{Z\in\mathfrak k_A:L_Z\ne\mathfrak g_2\},\qquad
\mathcal G=\mathfrak k_A\setminus\mathcal N.
\]
We describe $\mathcal N$ through invariant vectors and associative planes.
Invariance under $\operatorname{Ad}_S$ reduces the possibilities to common fixed vectors, invariant associative planes, and generated algebras of dimension at most three.
The first two cases are separated further by the rank of $Z$.
Their union will be identified in \Cref{thm:complete-failure-set} and converted to scalar equations in \Cref{sec:scalar}.

\paragraph{How the classification is organized.}
There are two different notions of dimension here: $\operatorname{rank}Z$ is the rank of a $7\times7$ skew matrix, whereas $\dim L_Z$ is the dimension of the Lie algebra obtained by taking all brackets.  The former does not determine the latter.  For nonzero $Z\in\mathfrak k_A$, the three frequencies in \Cref{subsec:stabilizer-coordinates} show that the rank is either four or six.  If $q_a=0$, the frequencies are $0,\beta,\beta$; if $q_a\ne0$, the only possible vanishing frequency is $|\beta-\alpha|$.

The maximal-subalgebra list first gives three possible containers for a proper $L_Z$: a vector stabilizer, a quaternionic stabilizer, or a principal $\mathfrak{su}(2)$.  Stability under $\Ad_S$ makes this reduction effective.  In the vector case the common kernel is $S$-invariant, so it contains an $S$-eigenvector.  In the quaternionic case, if $\dim L_Z>3$, the intersection bound forces the preserved quaternionic subalgebra itself to be $S$-invariant.  All remaining possibilities are handled by the low-dimensional equations, including the potential principal case.  These assertions are proved in \Cref{prop:failure,prop:maximal-reduction,thm:low-complete}; they are the reason the following case list is exhaustive.
\begin{center}
\renewcommand{\arraystretch}{1.18}
\begin{tabular}{@{}>{\raggedright\arraybackslash}p{5.1cm}>{\raggedright\arraybackslash}p{7.7cm}>{\raggedright\arraybackslash}p{3cm}@{}}\toprule
Case to resolve & Geometric information that determines the calculation & Result\\\midrule
Common fixed vector & An $S$-eigenline lies in $\ker Z$; solve $Zw=0$ in $\mathfrak k_A$. & \Cref{prop:failure,prop:rankfour-complete}\\
Rank six, no common vector, $\dim L_Z>3$ & The kernel line $\RR q_a$ and its reflected line force the associative plane. & \Cref{thm:visible-contexts}\\
Rank four, no common vector, $\dim L_Z>3$ & The even part selects a positive line; the odd part supplies a negative direction. & \Cref{thm:rankfour-context-complete}\\
$\dim L_Z\le3$ & Solve closure of the orthogonal triple $Z^+,Z^-,[Z^+,Z^-]$, treating the commuting case separately. & \Cref{thm:low-complete}\\\bottomrule
\end{tabular}
\end{center}
Preserving a plane means that the plane is mapped into itself; it does not mean that every vector in it is fixed.  The example $D_{13}$ above separates these notions.  Finally, \Cref{thm:complete-failure-set} assembles the cases into four geometric families.  The six scalar systems in Theorem~C are an equation-based description of their union, not six disjoint geometric components.  In particular, system (vi) also includes some common-vector boundary cases, as shown in \Cref{sec:scalar}.

\subsection{The eigenspaces of \texorpdfstring{$S$}{S}}
Let
\[
H_S=\RR1\oplus\Span\left\{e_1,\frac{e_2+e_4}{\sqrt2},\frac{e_3+e_5}{\sqrt2}\right\}.
\]
The displayed imaginary frame is associative, and the fixed involution satisfies
\[
S|_{H_S}=I,\qquad S|_{H_S^\perp}=-I.
\]
Thus $S$ is the orthogonal involution that is $+I$ on the quaternionic subalgebra $H_S$ and $-I$ on $H_S^\perp$.

The $+1$ eigenspace of $\Ad_S$ on $\gtwo$ is $\mathfrak k_{H_S}$; write
$\mathfrak m_{H_S}=\{B\in\gtwo:\Ad_SB=-B\}$ for the $-1$ eigenspace.
The projections $Z^+$ and $Z^-$ belong to these two summands, of dimensions six and eight.
By \Cref{lem:general-symmetric-pair}, $L_Z$ is $\Ad_S$-invariant.

\subsection{Pairs with a common fixed vector}
A \emph{common fixed vector} for the generated infinitesimal action means a nonzero $w\in V$ annihilated by both generators: $Zw=(\operatorname{Ad}_S Z)w=0$.  Equivalently, all their exponential flows fix $w$.  

For $0\ne w\in V$, write
\[
\mathfrak g_2^w=\{A\in\mathfrak g_2:Aw=0\}.
\]
After rescaling $w$ to unit length this is the standard maximal stabilizer $\mathfrak{su}(3)$.  The $\Ad_S$-stability from \Cref{lem:general-symmetric-pair} turns containment in such a stabilizer into a common-kernel condition.

For $Z=Z(q_a,q_b)\in\mathfrak k_A$, write $\alpha=\|q_a\|$ and $\beta=\|q_b\|$.

We use the rotation frequencies defined in \Cref{subsec:stabilizer-coordinates}, specifying explicitly when they must be positive and pairwise distinct.  A rank stratum is the subset on which the rank of this seven-dimensional operator is fixed.
For each $0\ne w\in V$, define the annihilator in the prescribed stabilizer
\[
\mathcal N_w=\{Z\in\mathfrak k_A:Zw=0\}.
\]
\begin{proposition}[A common-kernel criterion and the rank-six case]
\label{prop:failure}
Let $Z\in\mathfrak k_A$.  The following conditions are equivalent:
\begin{enumerate}[label=\textnormal{(\roman*)},leftmargin=9mm]
\item $\Lie\{Z,\Ad_SZ\}\subseteq\gtwo^w$ for some $0\ne w\in V$;
\item $\ker Z\cap S(\ker Z)\ne\{0\}$;
\item $\ker Z$ contains a nonzero vector in $P_{H_S}$ or a nonzero vector in $H_S^\perp$.
\end{enumerate}
Its three rotation frequencies, allowing zeros, are
\[
2\alpha,\qquad |\beta-\alpha|,\qquad \beta+\alpha.
\]
On the rank-six stratum $\alpha\ne0$, $\beta\ne\alpha$, one has $\ker Z=\RR q_a\subset P_A$, and
\[
\left\{Z\in\mathfrak k_A:\operatorname{rank}Z=6,
\ \Lie\{Z,\Ad_SZ\}\subseteq\gtwo^w\text{ for some }w\ne0\right\}
=\mathcal N_{e_1}\cap\{\operatorname{rank}Z=6\},
\]
where
\[
\mathcal N_{e_1}
=\Span\left\{D_{23},D_{45},D_{46}-\tfrac12D_{13},D_{47}+\tfrac12D_{12}\right\}.
\]
\end{proposition}

\begin{proof}
Since $S=S^{-1}$,
\[
(\Ad_SZ)w=0\iff Z(Sw)=0\iff w\in S(\ker Z),
\]
which proves the equivalence of the first two conditions.  The common-kernel space is $S$-invariant and therefore contains an eigenvector of the involution, proving the third condition.  The differential of \eqref{eq:gamma} yields
\[
\chi_Z(\lambda)=\lambda(\lambda^2+4\alpha^2)
(\lambda^2+(\beta-\alpha)^2)(\lambda^2+(\beta+\alpha)^2).
\]
On rank six the unique kernel is $\RR q_a$.  Because
\[
P_A\cap P_{H_S}=\RR e_1,
\qquad P_A\cap H_S^\perp=0,
\]
the criterion reduces to $Ze_1=0$.  If
\[
Z=x_0D_{12}+x_1D_{13}+x_2D_{23}+x_3D_{45}+x_4D_{46}+x_5D_{47},
\]
then $Ze_1=0$ is exactly $x_5=2x_0$ and $x_4=-2x_1$, giving the displayed basis.
\end{proof}

For the rank-four common-vector calculation, fix
$Z'=D_{12}-3D_{13}+2D_{47}$.

\begin{proposition}[A ruled family with a common fixed vector]
\label{prop:rankfour-failure}
For $0\ne w=ae_6+be_7$, the linear space $\mathcal N_w$ has dimension three and is spanned by
\begin{align*}
&D_{45}-D_{23},\\
&3abD_{12}-(2a^2-b^2)D_{13}+(a^2+b^2)D_{46},\\
&-(a^2-2b^2)D_{12}-3abD_{13}+(a^2+b^2)D_{47}.
\end{align*}
The union over $[a:b]\in\RR P^1$ is a ruled four-dimensional cone of common-vector failures.  At $a=b=1$, the witness $Z'$ has
\[
\ker Z'=\Span\{e_2,e_4+e_5,e_6+e_7\},
\qquad Z'e_1=-12e_3,
\]
\end{proposition}

\begin{proof}
In the ordered basis $(D_{12},D_{13},D_{23},D_{45},D_{46},D_{47})$, the map $Z\mapsto Zw$ has exact rank three for $(a,b)\ne(0,0)$.  Solving it gives the three displayed generators.  The incidence space over $\mathbb{RP}^1$ has dimension four, giving an upper bound for the dimension of its image.  At the displayed element $Z'$, its kernel intersects $\Span\{e_6,e_7\}$ in exactly $\RR(e_6+e_7)$.  This uniqueness persists on a relatively open part of the incidence space because the restriction to that two-plane has rank one there.  The projection to the derivation therefore has zero-dimensional fibers on that open part, proving that its image has dimension four.  The stated kernel and $Z'e_1$ follow by substitution.
\end{proof}

The positive and negative eigenvectors used in the incidence description are parametrized by
\[
w_+(a,b,c)=ae_1+b(e_2+e_4)+c(e_3+e_5),
\]
\[
w_-(a,b,r,t)=a(e_2-e_4)+b(e_3-e_5)+re_6+te_7.
\]
Define the homogeneous cubic representatives
\begin{align*}
Z_+(a,b,c)={}&-c(3b^2+c^2)D_{12}+2b^3D_{13}
-2a(b^2+c^2)D_{23}-a(b^2+c^2)D_{45}\\
&+b(3c^2-b^2)D_{46}+c(c^2-3b^2)D_{47},
\end{align*}
and, with $\upsilon=a^2t-2abr-b^2t$,
\begin{align*}
Z_-(a,b,r,t)={}&(3a^2b+art+b^3+br^2+2bt^2)D_{12}\\
&-(2a^3+2ar^2+at^2+brt)D_{13}+\upsilon D_{23}+2\upsilon D_{45}\\
&+(a^3-3ab^2+ar^2-at^2+2brt)D_{46}\\
&+(3a^2b+2art-b^3-br^2+bt^2)D_{47}.
\end{align*}
These formulas fix the chosen projective representatives.

\begin{proposition}[Rank-four pairs with a common fixed vector]
\label{prop:rankfour-complete}
Let $Z\in\mathfrak k_A$ be nonzero of rank four.  Then
\[
\Lie\{Z,\Ad_SZ\}\subseteq\gtwo^w
\quad\text{for some }w\ne0
\]
if and only if $Z$ belongs to at least one of the following four explicitly parametrized families:
\begin{enumerate}[label=\textnormal{(\alph*)},leftmargin=9mm]
\item $\mathcal N_{e_1}\cap\{\operatorname{rank}Z=4\}$;
\item the cubic line $\RR Z_+(a,b,c)$ over an eigenline
\[
[w_+(a,b,c)]\in\RR P(P_{H_S}),
\qquad (b,c)\ne(0,0);
\]
\item the cubic line $\RR Z_-(a,b,r,t)$ over an eigenline
\[
[w_-(a,b,r,t)]\in\RR P(H_S^\perp),
\qquad (a,b)\ne(0,0);
\]
\item the three-plane $\mathcal N_{re_6+te_7}$ from \Cref{prop:rankfour-failure}, for $[r:t]\in\RR P^1$.
\end{enumerate}
Away from the exceptional eigenlines $[e_1]$ and $\RR P\Span\{e_6,e_7\}$, the annihilator fiber in $\mathfrak k_A$ is one-dimensional.  These four cases exhaust the rank-four common-vector condition.
\end{proposition}

\begin{proof}
Write an $S$-eigenvector as $w=p+ye_4$, with $p\in P_A$ and quaternion coefficient $y\in H_A$.  Since every $Z\in\mathfrak k_A$ is block diagonal for $P_A\oplus H_A^\perp$, the equation $Zw=0$ is equivalent, in the quaternionic coordinates of \eqref{eq:gamma}, to
\[
[q_a,p]=0,
\qquad q_by=yq_a.
\]
If $p\ne0$ and $y\ne0$, the first equation forces $q_a$ to be a scalar multiple of $p$, and the second then determines $q_b=yq_ay^{-1}$; the annihilator is therefore a line.  For a positive eigenvector, $y=0$ occurs only on $\RR e_1$, giving the four-plane $\mathcal N_{e_1}$.  For a negative eigenvector, $p=0$ occurs exactly on $\Span\{e_6,e_7\}$, where $q_a$ is free and the annihilator has dimension three.  Substitution gives the displayed cubic generators.

Every nonzero solution outside the $e_1$ fiber has a kernel vector with nonzero components in both invariant blocks, or a nonzero kernel in the four-dimensional normal block.  Skew symmetry then forces a kernel of dimension at least three.  Since a nonzero element of $\mathfrak g_2$ cannot have rank two \cite{ChemtovKarigiannis2022}, it has rank four.  Conversely, \Cref{prop:failure} says that every rank-four vector failure contains an $S$-eigenvector, which lies in one of the four cases above.
\end{proof}

\begin{remark}[Associative kernel reduction]
The kernel of every rank-four element of $\mathfrak g_2$ is an associative three-plane \cite{ChemtovKarigiannis2022}.  Writing $P_Z=\ker Z$, the proposition can equivalently be viewed as a complete parametrization of those associative kernel planes arising from $\mathfrak k_A$ that satisfy $P_Z\cap S(P_Z)\ne0$.  The invariant-plane case outside this locus is classified in \Cref{thm:rankfour-context-complete}, and \Cref{thm:low-complete} treats all low-dimensional generated algebras.
\end{remark}

\subsection{Pairs preserving an associative plane}
We first classify all $S$-invariant quaternionic subalgebras.  We then determine their intersections with $\mathfrak k_A$ for elements of rank six.

The following description is the complex-line form of the invariant-plane family described by Draper and Mart\'in-Gonz\'alez \cite{DraperMartin2025}.  We include the proof to fix the parametrization used below.

Let $P_{H_S}$ and $Q_{H_S}=H_S^\perp$ be the $+1$ and $-1$ eigenspaces of $S$ in $V$.  

\begin{lemma}[Invariant associative planes]
\label{thm:invariant-contexts}
An associative three-plane $R\subset V$ is $S$-invariant if and only if either
\[
R=P_{H_S},
\]
or
\[
R=\RR m\oplus\Pi,
\]
where $0\ne m\in P_{H_S}$ and $\Pi\subset Q_{H_S}$ is a complex line for the complex structure $C_m/\lVert m\rVert$.  Hence the nontrivial family is a $\CC P^1$-bundle over $\RR P^2$ and has real dimension four.
\end{lemma}

\begin{proof}
Write $P=P_{H_S}$ and $Q=H_S^\perp$.  If an associative three-plane $R$ is $S$-invariant, then
\[
R=(R\cap P)\oplus(R\cap Q).
\]
Let the two dimensions be $(r_+,r_-)$, so $r_++r_-=3$.  The case $(3,0)$ gives $R=P$.  The case $(0,3)$ is impossible: if $x,y\in R\subset Q$ are independent, multiplicativity of $S$ gives
\[
S(x\times y)=Sx\times Sy=(-x)\times(-y)=x\times y.
\]
Thus the nonzero vector $x\times y$ lies in $P$, whereas associativity of $R$ requires it to lie in $R\subset Q$.  The case $(2,1)$ is also impossible: two independent vectors of the associative plane $P$ have cross product in $P$ and generate its third quaternionic direction, so an associative plane containing them cannot have only a two-dimensional intersection with $P$.

It remains to consider $(1,2)$.  Write $R\cap P=\RR m$ and $\Pi=R\cap Q$.  For $x\in\Pi$, associativity of $R$ implies $m\times x\in\Pi$.  Since $C_m^2=-\lVert m\rVert^2 I$ on $Q$, the operator $C_m/\lVert m\rVert$ is a complex structure on $Q$, and $\Pi$ is a complex line.  Conversely, if $\Pi$ is such a complex line, choose a unit $x\in\Pi$ and set $y=(m/\lVert m\rVert)\times x$.  Then $(m/\lVert m\rVert,x,y)$ is an oriented associative triple, so $\RR m\oplus\Pi$ is associative.  The positive line is parametrized by $\RR P^2$ and, for fixed $m$, the complex lines in the complex two-space $(Q,C_m/\lVert m\rVert)$ form $\CC P^1$.
\end{proof}

The $\CC P^1$-bundle over $\RR P^2$ is understood as a real bundle.  Replacing $m$ by $-m$ conjugates the complex structure on the negative eigenspace and leaves its set of complex lines unchanged.

Write $q_a=ae_1+ce_2+se_3$.  For parameters with $n=c^2+s^2>0$, set
\[
\zeta=c+is,\qquad c_3=\operatorname{Re}(\zeta^3)=c^3-3cs^2,
\qquad s_3=\operatorname{Im}(\zeta^3)=3c^2s-s^3.
\]
Define the projective family of subalgebras and two elements of each stabilizer intersection by
\begin{equation}
H_{[c:s]}=\RR1\oplus\Span\{ce_2+se_3,ce_4+se_5,(c^2-s^2)e_6+2cs\,e_7\},
\qquad[c:s]\in\RR P^1,
\label{eq:projective-context-family}
\end{equation}
\begin{align}
\mathsf P_{c,s}&=snD_{12}+c_3D_{46}+s_3D_{47},\notag\\
\mathsf Q_{c,s}&=cnD_{13}+c_3D_{46}+s_3D_{47}.
\label{eq:projective-intersection-basis}
\end{align}
For the resonant family, define
\begin{align*}
u_{a,c,s}&=2ae_1+c(e_2+e_4)+s(e_3+e_5),\\
v_{c,s}&=c(e_2-e_4)+s(e_3-e_5),\\
h_{a,c,s}&=-as(e_2-e_4)+ac(e_3-e_5)+(s^2-c^2)e_6-2cs\,e_7,
\end{align*}
\begin{equation}
H_{a;c,s}=\RR1\oplus\Span\{u_{a,c,s},v_{c,s},h_{a,c,s}\}.
\label{eq:resonant-context-family}
\end{equation}
Define the following derivation, polynomial in $(a,c,s)$:
\[
\widehat Z_{a;c,s}
=2c^2sD_{12}-c(c^2-s^2)D_{13}
 +an(D_{23}+D_{45})+c_3D_{46}+s_3D_{47}.
\]
These subspaces are shown below to be quaternionic subalgebras.  To state their
generation certificates, use the following six-word pattern for an element $U$:
\[
F_1(U)=U,\quad F_2(U)=\Ad_SU,\quad F_3(U)=[F_1(U),F_2(U)],
\]
\[
F_4(U)=[F_1(U),F_3(U)],\quad
F_5(U)=[F_2(U),F_3(U)],\quad
F_6(U)=[F_1(U),F_4(U)].
\]
Arguments will be suppressed when the element is specified.

\begin{theorem}[Rank-six pairs preserving an associative plane]
\label{thm:visible-contexts}
Let $Z\in\mathfrak k_A$ have rank six and suppose
$L_Z\subseteq\mathfrak k_H$ with $\dim L_Z>3$ for a quaternionic subalgebra $H$.
Then $S(H)=H$.  If $Z\notin\mathcal N_{e_1}$, so that $n>0$, precisely the following cases occur.
\begin{enumerate}[label=\textnormal{(\roman*)},leftmargin=9mm]
\item If $a=0$, then $H=H_{[c:s]}$ and
\[
\mathfrak k_A\cap\mathfrak k_{H_{[c:s]}}
 =\Span\{\mathsf P_{c,s},\mathsf Q_{c,s}\}.
\]
\item If $a\ne0$, then $H=H_{a;c,s}$ and the second quaternionic coordinate is forced to be
\begin{equation}
q_b=3\left(ae_1+\frac{c_3}{n} e_2+\frac{s_3}{n} e_3\right).
\label{eq:resonant-qb}
\end{equation}
Equivalently,
\begin{equation}
nZ=\widehat Z_{a;c,s}.
\label{eq:resonant-pulse}
\end{equation}
The element $\widehat Z_{a;c,s}$ satisfies $\beta=3\alpha$ and generates the full stabilizer:
\[
\Lie\{\widehat Z_{a;c,s},\Ad_S\widehat Z_{a;c,s}\}
 =\mathfrak k_{H_{a;c,s}}\cong\mathfrak{so}(4).
\]
Its six-word trace-Gram certificate is
\begin{equation}
\det\bigl(-\operatorname{tr}(F_i(\widehat Z)F_j(\widehat Z))\bigr)_{i,j=1}^{6}
 =2^{46}3^{11}n^{34}(a^2+n)^2(2a^2+n)^6\ne0.
\label{eq:resonant-so4-gram}
\end{equation}
\end{enumerate}
Conversely, every element in the projective intersection of (i) and every element
$\widehat Z_{a;c,s}$ preserves the stated $S$-invariant quaternionic subalgebra.
The intersection formula is valid at every projective point, and the homogeneous certificate
also holds at $a=0$.
\end{theorem}

\begin{proof}
Because $L_Z$ is $\Ad_S$-stable,
\[
L_Z\subseteq\mathfrak k_H\cap\mathfrak k_{S(H)}.
\]
The hypothesis $\dim L_Z>3$ and \Cref{lem:nonfixed-context-bound} force $S(H)=H$.  The possible invariant associative planes are therefore exactly those described in \Cref{thm:invariant-contexts}.  Since $Z$ has rank six, its kernel is the line $\RR q_a$.  The restriction of a skew operator to the invariant odd-dimensional plane $P_H$ has a zero eigenvalue, so $q_a\in P_H$.  Hence $S(q_a)\in P_H$ as well.  Outside $\mathcal N_{e_1}$ these two vectors are independent and, by alternativity, determine the unique quaternionic subalgebra generated by them.  

If $a=0$, the two vectors are $ce_2+se_3$ and $ce_4+se_5$; their cross product is $(c^2-s^2)e_6+2cs\,e_7$, giving \eqref{eq:projective-context-family}.  For
$Z=x_0D_{12}+x_1D_{13}+x_2D_{23}+x_3D_{45}+x_4D_{46}+x_5D_{47}$,
using the normal frame
$e_1,-se_2+ce_3,-se_4+ce_5,-2cs\,e_6+(c^2-s^2)e_7$
reduces $Z(P_{H_{[c:s]}})\subseteq P_{H_{[c:s]}}$ to
\[
\begin{pmatrix}
-4c&-4s&0&0&-2s&2c\\
0&0&4n&-2n&0&0\\
0&0&-2n&4n&0&0\\
0&0&0&0&3s(3c^2-s^2)&-3c(c^2-3s^2)
\end{pmatrix}
\begin{pmatrix}x_0\\x_1\\x_2\\x_3\\x_4\\x_5\end{pmatrix}=0.
\]
The middle two rows have determinant $12n^2$ in columns $x_2,x_3$.
The first row is nonzero in columns $x_0,x_1$, and the squared norm of the last row is $9n^3$.
Thus the rank is four for every $n>0$.
Both $\mathsf P_{c,s}$ and $\mathsf Q_{c,s}$ solve the system.  They are independent because
their last two coefficients are $(c_3,s_3)\ne(0,0)$ and
$\mathsf P_{c,s}-\mathsf Q_{c,s}=n(sD_{12}-cD_{13})\ne0$.
They therefore span the full intersection.

Now suppose $a\ne0$.  The sum and difference of $q_a$ and $S(q_a)$ are $u_{a,c,s}$ and $v_{c,s}$.  Direct multiplication gives
\[
u_{a,c,s}\times v_{c,s}=2h_{a,c,s},
\quad
v_{c,s}\times h_{a,c,s}=n\,u_{a,c,s},
\quad
h_{a,c,s}\times u_{a,c,s}=(2a^2+n)v_{c,s}.
\]
Thus \eqref{eq:resonant-context-family} is associative and is fixed by $S$, with $S(u)=u$, $S(v)=-v$, and $S(h)=-h$.

Write $q_b=b_1e_1+b_2e_2+b_3e_3$.  Substitution in the invariance equations for this forced quaternionic subalgebra first yields
\[
b_1=3a,
\]
and then
\[
\begin{pmatrix}c^2-s^2&2cs\\2cs&-(c^2-s^2)\end{pmatrix}
\begin{pmatrix}b_2\\b_3\end{pmatrix}
=
\begin{pmatrix}3cn\\-3sn\end{pmatrix}.
\]
The determinant is $-n^2$, so the solution is unique and equals \eqref{eq:resonant-qb}.  Inverting the linear $(q_a,q_b)$ coordinate map gives \eqref{eq:resonant-pulse}.  In particular,
\[
\alpha^2=a^2+n,
\qquad
\beta^2=9(a^2+n),
\]
for $Z=\widehat Z/n$, and hence $\beta=3\alpha$.

For the six words $F_i$ displayed above, Appendix~\ref{app:resonant-blocks} derives \eqref{eq:resonant-so4-gram} from three $2\times2$ Gram blocks.  It first uses the residual circle and homogeneity to put $(c,s)=(1,0)$, computes the three blocks there, and restores the scale; this explains both the factorization and the power $n^{34}$.  All factors are positive for $n>0$.  Hence the six words are independent and the generated algebra is the full six-dimensional stabilizer $\mathfrak k_{H_{a;c,s}}$.

For the converse, the polynomial representative satisfies
\[
\widehat Z u=-2nh,
\qquad
\widehat Z v=2nh,
\qquad
\widehat Z h=n^2u-n(2a^2+n)v.
\]
It therefore preserves $H_{a;c,s}$, and so does its $S$-conjugate, proving nongeneration.  At $a=0$, the quaternionic subalgebra reduces to $H_{[c:s]}$ and
\[
\widehat Z_{0;c,s}
=\frac{2c^2}{n}\mathsf P_{c,s}
 -\frac{c^2-s^2}{n}\mathsf Q_{c,s},
\]
so the resonant cone attaches to the projective component along a distinguished element line.
\end{proof}

The projective endpoints $[1:0]$ and $[0:1]$ give $H_{246}$ and $H_{365}$;
write $H_\star=H_{[1:1]}$ for the noncoordinate member.  The remaining coordinate
fixed quaternionic subalgebra satisfies
\[
\mathfrak k_A\cap\mathfrak k_{176}=\Span\{D_{23},D_{45}\}\subset\mathcal N_{e_1}.
\]
Individual intersection generators may have rank below six.  More precisely,
\[
\alpha^2(\mathsf P_{c,s})=n(c^2-3s^2)^2,\qquad
\alpha^2(\mathsf Q_{c,s})=n(3c^2-s^2)^2.
\]
Since $n>0$, the respective zero sets are $[c:s]=[\sqrt3:\pm1]$ and
$[c:s]=[1:\pm\sqrt3]$.  All four points belong to $\mathcal N_{e_1}$.

\begin{corollary}[The resonant rank-six family]
\label{cor:resonant-cone}
The homogeneous family $\widehat Z_{a;c,s}$, with $n>0$, has real affine dimension
three and projective dimension two.  Its members with $a\ne0$ lie outside
$x_2=x_3=0$ and preserve proper quaternionic stabilizers on the rank-six resonance
$\beta=3\alpha$.
\end{corollary}
\begin{proof}
Every coefficient of $\widehat Z_{a;c,s}$ is homogeneous of degree three in $(a,c,s)$.
The last two coefficients obey $x_4+ix_5=(c+is)^3$, so a nonzero image has only
finitely many choices of $(c,s)$.  Once they are fixed, $x_2=x_3=an$ determines $a$.
Thus the three-dimensional parameter domain has finite fibers and its image has
real dimension three.  Nonzero real scaling acts by scaling all parameters, so
projectivization has dimension two.  If $a\ne0$ then $x_2=x_3\ne0$; the preservation
and resonance claims follow from \Cref{thm:visible-contexts}.
\end{proof}

Write
\[
Z=x_0D_{12}+x_1D_{13}+x_2D_{23}+x_3D_{45}+x_4D_{46}+x_5D_{47}
\]
For complex scalars, a superscript $\ast$ denotes complex conjugation. Introduce
\[
w=x_0+ix_1,
\qquad z=x_4+ix_5,
\qquad z^\ast=x_4-ix_5,
\qquad \omega=z-2iw.
\]
Here $z^\ast$ denotes the complex conjugate of $z$.

\begin{corollary}[An equation for the projective family]
\label{cor:quartic-cone}
$\mathcal N_{e_1}=\{\omega=0\}$.  In the four-plane $x_2=x_3=0$, the union of the projective stabilizer-intersection planes is exactly
\begin{equation}
\operatorname{Im}(z^\ast\omega^3)=0.
\label{eq:so4-quartic-cone}
\end{equation}
It is a real homogeneous quartic cone of dimension three.  Its singular locus is $\omega=0$, a real two-plane in the ambient four-plane.
\end{corollary}

\begin{proof}
For the basis in \eqref{eq:projective-intersection-basis}, one has
\[
z=\mu\zeta^3,
\qquad \omega=\lambda\zeta
\]
with real $\lambda,\mu$, so \eqref{eq:so4-quartic-cone} follows.  Conversely, if $\omega\ne0$, choose a projective representative $\zeta$ in its real direction.  The quartic equation forces $z$ to have the phase of $\zeta^3$ up to sign.  For an element $\xi\mathsf P_{c,s}+\eta\mathsf Q_{c,s}$, its two real coefficients satisfy
\[
\begin{pmatrix}\lambda\\\mu\end{pmatrix}
=\begin{pmatrix}c^2-3s^2&3c^2-s^2\\1&1\end{pmatrix}
\begin{pmatrix}\xi\\\eta\end{pmatrix}.
\]
The coefficient matrix has determinant $-2(c^2+s^2)\ne0$, so the required
combination exists.  When $\omega=0$ and $z\ne0$, choose $\zeta\ne0$ with the phase of a cube root of $z$, and take $\lambda=0$ and real $\mu$ satisfying $z=\mu\zeta^3$.  The same invertible coefficient map supplies the required linear combination.  If $z=\omega=0$, the element is zero.  This proves the converse also on the singular locus.

For the singular locus put $F=\operatorname{Im}(z^\ast\omega^3)$.  Direct differentiation in the original real coordinates gives
\[
\frac{\partial F}{\partial x_0}=-6\operatorname{Re}(z^\ast\omega^2),
\qquad
\frac{\partial F}{\partial x_1}=6\operatorname{Im}(z^\ast\omega^2).
\]
Thus a singular point satisfies $z^\ast\omega^2=0$.  If $\omega=0$, all first derivatives vanish.  If instead $z=0$ and $\omega\ne0$, then
\[
\left.\frac{\partial F}{\partial x_4}\right|_{z=0}=\operatorname{Im}(\omega^3),
\qquad
\left.\frac{\partial F}{\partial x_5}\right|_{z=0}=-\operatorname{Re}(\omega^3),
\]
which cannot both vanish.  Hence the singular locus is exactly $\omega=0$.
\end{proof}

\subsection{Reduction to low-dimensional algebras}
The compact classification recalled in \Cref{subsec:compact-maximals} gives the following reduction for the fixed involution.

\begin{proposition}[Reduction to invariant subspaces]
\label{prop:maximal-reduction}
Let
\[
L_Z=\Lie\{Z,\Ad_SZ\}\subsetneq\gtwo.
\]
Then one of the following holds:
\begin{enumerate}[label=\textnormal{(\roman*)},leftmargin=9mm]
\item $L_Z\subseteq\gtwo^w\cong\mathfrak{su}(3)$ for an $S$-eigenvector $w$;
\item $L_Z\subseteq\mathfrak k_H\cong\mathfrak{so}(4)$ for an $S$-invariant quaternionic subalgebra $H$;
\item $\dim L_Z\le3$.
\end{enumerate}
In particular, if $\dim L_Z>3$, then $L_Z$ fixes a nonzero vector or preserves an $S$-invariant associative plane.
\end{proposition}

\begin{proof}
By the compact maximal-subalgebra classification in \Cref{subsec:compact-maximals}, a proper generated algebra is contained in a maximal subalgebra of vector-stabilizer, quaternionic-stabilizer, or principal-$\mathfrak{su}(2)$ type.  Because $L_Z$ is $\Ad_S$-stable, vector containment in $\gtwo^w$ also gives containment in $\gtwo^{Sw}$.  If $Sw\not\parallel w$, their intersection is the stabilizer of two independent vectors and has dimension three; otherwise (i) holds.

Likewise, containment in $\mathfrak k_H$ gives containment in $\mathfrak k_H\cap\mathfrak k_{S(H)}$.  If $S(H)\ne H$, \Cref{lem:nonfixed-context-bound} gives dimension at most three; otherwise (ii) holds.  A principal $\mathfrak{su}(2)$ is itself three-dimensional, so it belongs to case (iii).
\end{proof}

\subsection{Solving the low-dimensional equations}
It remains to determine when $\dim L_Z\le3$.  A one-parameter subgroup centralizing $S$ preserves $\mathfrak k_A$ and allows one angular coordinate to be normalized.  We solve the resulting real equations separately when the first transverse coordinate is zero and when it is nonzero.

Use the coordinates $(a_1,a_2,a_3,b_1,b_2,b_3)$ from \eqref{eq:terminal-qa-qb}, and put
\[
q=a_2+ia_3,\qquad p=b_2+ib_3.
\]
Thus $q=-\omega$ and $p=3z$ in the notation of \Cref{cor:quartic-cone}.
  Let
\[
 \Xi=\frac12(D_{23}+D_{45}),\qquad g_\theta=\exp(\theta\Xi).
\]
The transformation $\Ad_{g_\theta}$ preserves $\mathfrak k_A$ and commutes with $\Ad_S$.  Direct differentiation of \eqref{eq:terminal-qa-qb} gives
\begin{equation}
 (a_1,b_1,q,p)\longmapsto(a_1,b_1,e^{i\theta}q,e^{3i\theta}p).
 \label{eq:residual-circle}
\end{equation}
Both conjugation by $g_\theta$ and nonzero real rescaling preserve the dimension and type of $L_Z$.

The weights one and three mean exactly the factors $e^{i\theta}$ and $e^{3i\theta}$ in \eqref{eq:residual-circle}.  A normal form below is taken up to this circle action and nonzero real rescaling of $Z$; these operations conjugate the generated algebra or leave it unchanged.  Scalar coordinates $a,b,c,d$ in a normalized pair such as $q_a=(a,1,0)$, $q_b=(d,b,c)$ are local to that normal-form calculation.

\begin{theorem}[Generated algebras of dimension at most three]
\label{thm:low-complete}
Let $Z\in\mathfrak k_A$ and $L_Z=\Lie\{Z,\Ad_SZ\}$.
\begin{enumerate}[label=\textnormal{(\roman*)},leftmargin=9mm]
\item The generated algebra is abelian if and only if
\[
 x_0=x_1=x_4=x_5=0.
\]
On this plane $Z=x_2D_{23}+x_3D_{45}$.  Its generated algebra has dimension zero at the origin, dimension one on the two nonzero lines $x_2=x_3$ and $x_2=-x_3$, and dimension two elsewhere.
\item Outside this plane, $\dim L_Z\le3$ if and only if $a_1=b_1=0$ and either
\[
 q=0,\quad p\ne0,
\]
or
\begin{equation}
 q\ne0,\qquad |q|^2p=\eta q^3,
 \qquad \eta\in\{1,-1,-3\}.
 \label{eq:three-small-families}
\end{equation}
Every generated algebra in (ii) is isomorphic to $\mathfrak{su}(2)$ and has dimension exactly three.
\end{enumerate}
\end{theorem}
\begin{proof}
Put $A=(Z+\Ad_SZ)/2$, $B=(Z-\Ad_SZ)/2$, and $W=[A,B]$.  Substitution of the matrices of $A$ and $B$ gives
\[
 W_{24}+W_{35}=12(x_0^2-x_0x_5+x_5^2+x_1^2+x_1x_4+x_4^2).
\]
The quadratic form is positive definite in $(x_0,x_1,x_4,x_5)$.  Thus $W=0$ forces those four coordinates to vanish; conversely $D_{23}$ and $D_{45}$ commute and are exchanged by $\Ad_S$.  This proves (i).

If $W\ne0$, the three vectors $A,B,W$ are nonzero and pairwise trace-orthogonal.  A three-dimensional generated algebra must therefore equal their span.  Parity under $\Ad_S$ and invariance of the trace metric show that this is equivalent to
\[
[A,W]=-kB,\qquad [B,W]=lA
\]
for real $k,l>0$.  Conversely these relations close the span and give a compact simple three-dimensional algebra.

When $q\ne0$, use \eqref{eq:residual-circle} and a nonzero scalar rescaling to normalize
\[
 q_a=(a,1,0),\qquad q_b=(d,b,c).
\]
In the ordered basis $(e_1,\ldots,e_7)$, this derivation is
\[
Z=\begin{pmatrix}
0&0&2&0&0&0&0\\
0&0&-2a&0&0&0&0\\
-2&2a&0&0&0&0&0\\
0&0&0&0&a-d&1-b&-c\\
0&0&0&-a+d&0&-c&b+1\\
0&0&0&b-1&c&0&-a-d\\
0&0&0&c&-b-1&a+d&0
\end{pmatrix}.
\]
Using this matrix and $S$ from \eqref{eq:switch}, the matrices $U=[A,W]+kB$ and $V_0=[B,W]-lA$ have entries
\[
 U_{12}=3c,\quad (V_0)_{25}=-3\{a(b^2+c^2+2)-d\},\quad
 U_{16}=-2\{3ab+(b+2)d\}.
\]
Their vanishing implies $c=0$, $d=a(b^2+2)$, and
\[
 a(b+1)(b^2+b+4)=0.
\]
The final factor is strictly positive for real $b$.  If $b=-1$, then $d=3a$ and $U_{37}=16a^2$, so $a=d=0$ in all cases.  With $a=d=c=0$, the $(1,3)$ entries give
$k=l=(3b^2+6b+7)/2$, and the $(2,6)$ entry becomes
\[
U_{26}=-\tfrac12(b-1)(b+1)(b+3).
\]
All residual entries vanish at $b=1,-1,-3$, proving sufficiency for these three roots.  Undoing the normalization gives \eqref{eq:three-small-families}.

When $q=0$ but $p\ne0$, normalize instead to $q_a=(a,0,0)$, $q_b=(d,1,0)$.  Now $(V_0)_{25}=-3a$ forces $a=0$.
The $(2,6)$ entries give $k=(3d^2+1)/2$ and $l=(1-d^2)/2$; the $(2,3)$ entry of $U$ then gives $d(d^2+1)=0$.
Thus $d=0$.  Substitution again gives a nonabelian three-dimensional generated algebra.  The elimination is over the real numbers.
\end{proof}

\begin{corollary}[No principal generated algebra for $S$]
\label{cor:no-principal}
No nonzero $Z\in\mathfrak k_A$ has $L_Z$ contained in a principal $\mathfrak{su}(2)$ of $\mathfrak g_2$.
\end{corollary}
\begin{proof}
A principal $\mathfrak{su}(2)$ acts irreducibly on $V$ and every nonzero element has three distinct positive frequencies in the ratio $1:2:3$.  The nonabelian cases of \Cref{thm:low-complete} instead have rank four, or rank six with frequencies $1:1:2$.  The nonzero one-dimensional cases in (i) have the same alternatives.  An abelian two-dimensional algebra cannot lie in $\mathfrak{su}(2)$, and a generated algebra of dimension greater than three cannot lie there either.
\end{proof}

\subsection{Rank-four pairs without a common fixed vector}
We now consider rank-four $Z$ for which $Z$ and $SZS$ have no common fixed vector.  Although $\ker Z$ is three-dimensional, the even part of $Z$ determines the positive line of any $S$-invariant associative plane preserved by the pair.  This reduces the classification to the one-parameter family defined below, together with rescaling and circle conjugation.

For $0<\tau<1$, put
\begin{equation}
 \nu_\tau=\tau^2+3,\qquad
 \kappa^2=\frac{(1-\tau)\nu_\tau}{\tau},\qquad
 h=-\frac{\tau(\tau+1)}{\nu_\tau}\kappa,
 \label{eq:rankfour-parameter}
\end{equation}
where either real sign of $\kappa$ is allowed.  The two signs satisfy
\[
Z_{\tau,-\kappa}=-\Ad_{g_\pi}Z_{\tau,\kappa}.
\]
Define
\begin{align}
 Z_{\tau,\kappa}={}&(\tau-3)\nu_\tau D_{13}
 +(\tau^2-\tau+6)\kappa D_{23}\notag\\
 &+(2\tau^2+\tau+3)\kappa D_{45}
 -2\tau \nu_\tau D_{46},
 \label{eq:rankfour-new-pulse}\\
 m={}&\frac\kappa2e_1+e_2+e_4,\qquad
 v=h(e_3-e_5)-(\tau+1)e_6,\qquad \ell=m\times v.
 \label{eq:rankfour-new-frame}
\end{align}
Set $H_{\tau,\kappa}=\RR1\oplus\Span\{m,v,\ell\}$.

\begin{theorem}[Rank-four pairs preserving an associative plane]
\label{thm:rankfour-context-complete}
Every $Z_{\tau,\kappa}$ has rank four, its kernel has zero intersection with the kernel of $SZ_{\tau,\kappa}S$, and
\[
 \Lie\{Z_{\tau,\kappa},\Ad_SZ_{\tau,\kappa}\}
 =\mathfrak k_{H_{\tau,\kappa}}\cong\mathfrak{so}(4).
\]
Conversely, if $Z\in\mathfrak k_A$ has rank four, $\ker Z\cap S\ker Z=0$, and $L_Z$ is contained in an $S$-invariant quaternionic stabilizer, then
\[
 Z=\lambda\Ad_{g_\theta}Z_{\tau,\kappa}
\]
for some $\lambda\ne0$, $\theta\in\RR$, and parameters satisfying \eqref{eq:rankfour-parameter}.  Thus this case is a three-real-parameter family after allowing rescaling and the residual circle.
\end{theorem}
\begin{proof}
Outside the common-vector locus, $q\ne0$: otherwise $Ze_1=0$.  Normalize to $q_a=(a,1,0)$ and $q_b=(d,b,c)$ using \eqref{eq:residual-circle}.  Rank four is equivalent to
\begin{equation}
 d^2+b^2+c^2=a^2+1.
 \label{eq:rankfour-sphere}
\end{equation}
Write $A=(Z+\Ad_SZ)/2$ and $B=(Z-\Ad_SZ)/2$.  On the positive eigenspace of $S$, in the ordered frame $(e_1,e_2+e_4,e_3+e_5)$, the even part $A$ has matrix
\[
 \begin{pmatrix}0&0&2\\0&0&-(a+d)/2\\-1&(a+d)/2&0\end{pmatrix}.
\]
The frame is not orthonormal: its Gram matrix is $\operatorname{diag}(1,2,2)$.  The unique kernel line is generated by $m=(a+d)e_1/2+e_2+e_4$.  Every invariant quaternionic subalgebra must have this positive line; the exceptional all-positive quaternionic subalgebra would give a common fixed vector and is excluded.  Its negative plane contains
\[
 v=Bm=(a-d)(e_3-e_5)+(b-1)e_6+ce_7.
\]
This vector is nonzero, since otherwise $m$ would be a common kernel vector.  Associativity forces the entire quaternionic subalgebra to be $\RR1+\Span\{m,v,m\times v\}$.

Preservation by $B$ requires $Bv\parallel m$.  Since $m$ has zero $e_3$ component and $(Bv)_3=c$, its coordinate equations give $c=0$.  With $h=a-d$, $k=a+d$, the remaining equation and \eqref{eq:rankfour-sphere} reduce to
\[
hk=b^2-1,\qquad (b^2+3)h+b(b-1)k=0.
\]
If $b=1$, these equations force $h=0$, so $v=0$.  The value $b=0$ is impossible.  For $b\ne0,1$ they imply
\[
 k^2=-\frac{(b+1)(b^2+3)}b.
\]
Real solutions have $-1\le b<0$.  At $b=-1$, $k=h=0$, and \Cref{thm:low-complete} gives a common-vector generated algebra.  The remaining interval is exactly $b=-\tau$, $0<\tau<1$.  Solving for $a,d$ yields
\[
 a=\frac{3-\tau}{2\nu_\tau}\kappa,\qquad
 d=\frac{2\tau^2+\tau+3}{2\nu_\tau}\kappa.
\]
Multiplying the normalized element by $6\nu_\tau$ gives \eqref{eq:rankfour-new-pulse}.  

For sufficiency, use the even and odd parts $A_{\tau,\kappa},B_{\tau,\kappa}$ of $Z_{\tau,\kappa}$.  Exact substitution into \eqref{eq:rankfour-new-frame} gives
\[
 A_{\tau,\kappa}m=0,\quad B_{\tau,\kappa}m=6\nu_\tau v,\quad
 A_{\tau,\kappa}v=-\frac{12\nu_\tau}{\tau+1}\ell,\quad
 B_{\tau,\kappa}v=-24\tau(\tau+1)m.
\]
The derivation rule then supplies the action on $\ell=m\times v$.  Hence both generators preserve the same associative plane.  They have rank four by \eqref{eq:rankfour-sphere}.  To prove full six-dimensional generated algebra, take
\[
 F=(A_{\tau,\kappa},B_{\tau,\kappa},[A_{\tau,\kappa},B_{\tau,\kappa}],[A_{\tau,\kappa},[A_{\tau,\kappa},B_{\tau,\kappa}]],
 [B_{\tau,\kappa},[A_{\tau,\kappa},B_{\tau,\kappa}]],[A_{\tau,\kappa},[A_{\tau,\kappa},[A_{\tau,\kappa},B_{\tau,\kappa}]]]).
\]
Appendix~\ref{app:rankfour-blocks} gives the intermediate calculation: after normalizing $Z_{\tau,\kappa}$ by $6\nu_\tau$, parity and trace invariance split the Gram matrix into three $2\times2$ blocks.  Multiplying their determinants and restoring the scale gives
\begin{equation}
 \det\operatorname{Gram}_{\rm tr}(F)=
 \frac{2^{32}3^{31}(3-\tau)^{10}(1-\tau)^6(1+\tau)^{10}
 \nu_\tau^{24}(\tau^2-2\tau+3)^4}{\tau^6}>0.
 \label{eq:rankfour-new-gram}
\end{equation}
The displayed determinant is strictly positive for $0<\tau<1$, so the six listed words are independent.  Since they all lie in the six-dimensional stabilizer, the generated algebra is the full stabilizer.  Its $3\oplus4$ representation on $V$ has no fixed vector, proving the common-kernel assertion as well.
\end{proof}

For example, at $\tau=1/2$ and $\kappa=\sqrt{13}/2$, an integral-quadratic rescaling gives
\[
Z_{\rm new}=-65D_{13}+23\sqrt{13}D_{23}
             +16\sqrt{13}D_{45}-26D_{46}.
\]
It satisfies $\alpha^2=\beta^2=36036$, has rank four and no common vector with $\Ad_SZ_{\rm new}$, and generates a full $\mathfrak{so}(4)$.  This is an explicit rank-four quaternionic-stabilizer failure without a common fixed vector.
As $\tau\to1^-$ one has $\kappa\to0$ and
\[
Z_{\tau,\kappa}\longrightarrow-8(D_{13}+D_{46}),
\]
which is the $\eta=-1$ common-vector representative from the low-dimensional classification.  As $\tau\to0^+$, $|\kappa|\to\infty$ in the normalization \eqref{eq:rankfour-parameter}.

We now assemble the four failure families.  Define
\begin{align*}
\mathcal F_{\rm vec}&=\{Z\in\mathfrak k_A:\ker Z\cap S(\ker Z)\ne\{0\}\},\\
\mathcal F_{\rm proj}&=\bigcup_{[c:s]\in\RR P^1}
 (\mathfrak k_A\cap\mathfrak k_{H_{[c:s]}}),\\
\mathcal F_{\rm res}&=\{\widehat Z_{a;c,s}:a,c,s\in\RR,\ c^2+s^2>0\}\cup\{0\},\\
\mathcal F_{\rm rf}&=\{\lambda\Ad_{g_\theta}Z_{\tau,\kappa}:\lambda\in\RR\setminus\{0\},\
\theta\in\RR,\ 0<\tau<1,\
\kappa^2=(1-\tau)(\tau^2+3)/\tau\}.
\end{align*}
The common-vector set is described in \Cref{prop:failure,prop:rankfour-complete},
the projective set in \Cref{cor:quartic-cone}, the resonant set in
\Cref{cor:resonant-cone}, and the last set in \Cref{thm:rankfour-context-complete}.
The four families may intersect.

\begin{theorem}[The four geometric families]
\label{thm:complete-failure-set}
For the fixed pair $(\mathfrak k_A,S)$, the complete nongenerating locus is
\begin{equation}
\{Z\in\mathfrak k_A:L_Z\ne\mathfrak g_2\}
 =\mathcal F_{\rm vec}\cup\mathcal F_{\rm proj}\cup\mathcal F_{\rm res}\cup\mathcal F_{\rm rf}.
\label{eq:complete-failure-set}
\end{equation}
\end{theorem}
\begin{proof}
First consider the inclusion from right to left.  The common-vector criterion
\Cref{prop:failure} places $\mathcal F_{\rm vec}$ in a vector stabilizer.
\Cref{thm:visible-contexts} places $\mathcal F_{\rm proj}$ and the nonzero part of
$\mathcal F_{\rm res}$ in quaternionic stabilizers.  Finally,
\Cref{thm:rankfour-context-complete} does the same for $\mathcal F_{\rm rf}$.
All of these containing algebras are proper, and $Z=0$ plainly fails to generate.

For the reverse inclusion, let $L_Z\ne\mathfrak g_2$.  If
$\dim L_Z>3$, \Cref{prop:maximal-reduction} places it in a vector stabilizer or
an $S$-invariant quaternionic stabilizer.  The vector case gives
$Z\in\mathcal F_{\rm vec}$.  In the other case we may assume there is no common
fixed vector; in particular $Z\notin\mathcal N_{e_1}$.  The nonzero element $Z$
has rank four or six.  At rank six, \Cref{thm:visible-contexts} gives
$Z\in\mathcal F_{\rm proj}\cup\mathcal F_{\rm res}$; the rescaling in
\eqref{eq:resonant-pulse} does not change this conical set.
At rank four, \Cref{thm:rankfour-context-complete} gives
$Z\in\mathcal F_{\rm rf}$.

If $\dim L_Z\le3$, apply \Cref{thm:low-complete}.
Its abelian plane and its $q=0$ family fix $e_1$, so they lie in
$\mathcal F_{\rm vec}$.  In the three remaining normal-form families,
$a_1=b_1=0$ and $|q|^2p=\eta q^3$ with $\eta\in\{1,-1,-3\}$.
Thus $x_2=x_3=0$ and $\operatorname{Im}(p(q^\ast)^3)=0$.
Since $q=-\omega$ and $p=3z$,
\[
\operatorname{Im}\bigl(p(q^\ast)^3\bigr)
 =3\operatorname{Im}(z^\ast\omega^3)=0.
\]
\Cref{cor:quartic-cone} therefore places these cases in $\mathcal F_{\rm proj}$.
The two inclusions cover all possibilities.
\end{proof}

The four families can be distinguished by the following elements:
\[
\begin{split}
Z_{\rm v}&=D_{23},\quad Z_{\rm p}=D_{13},\quad
Z_{\rm r}=-D_{13}+D_{23}+D_{45}+D_{46},\\
Z_{\rm f}&=-65D_{13}+23\sqrt{13}\,D_{23}+16\sqrt{13}\,D_{45}-26D_{46}.
\end{split}
\]
\begin{corollary}[Irredundance of the geometric families]\label{cor:geometric-irredundance}
None of the four families in \Cref{thm:complete-failure-set} is contained in the union of the other three.
\end{corollary}
\begin{proof}
The element $Z_{\rm v}$ fixes $e_1$, has rank six, and has zero transverse coordinate $q$ and nonzero axial coordinate $a_1$; it belongs only to $\mathcal F_{\rm vec}$.
For $Z_{\rm p}$, the kernel is $\RR e_2$ and its reflected kernel is $\RR e_4$, so there is no common fixed vector.  It lies in $\mathcal F_{\rm proj}$ by \Cref{cor:quartic-cone}, has rank six, and has $\beta=0$, excluding the other three families.
The element $Z_{\rm r}$ is $\widehat Z_{1;1,0}$: \Cref{thm:visible-contexts} places it in the resonant family, with $\beta=3\alpha$ and $a_1=1$.  Its rank is six, it has no common fixed vector, and the nonzero axial coordinate excludes the projective family.
Finally, $Z_{\rm f}$ is the rank-four witness following \Cref{thm:rankfour-context-complete}.  It has no common fixed vector and $a_1=30\sqrt{13}\ne0$, so it is neither vector nor projective.  Its $\alpha=\beta>0$ excludes the resonant family.  Thus each displayed witness belongs to its named family alone.
\end{proof}

\section{Scalar equations for the fixed nongenerating set}
\label{sec:scalar}

We eliminate the vector, plane, and circle parameters in the four families of \Cref{thm:complete-failure-set}.  This gives six systems of equations in the coordinates of $Z\in\mathfrak k_A$.  After proving their equivalence and irredundance, we combine them into one nonnegative homogeneous polynomial and determine the dimension of its zero set.

Keep the coordinates of \eqref{eq:terminal-qa-qb}, and abbreviate
\[
a=a_1,\qquad d=b_1,\qquad q=a_2+ia_3,\qquad p=b_2+ib_3.
\]
Thus $(a,d,q,p)\in\RR^2\times\CC^2$ are linear coordinates on $\mathfrak k_A$.  Set
\begin{align}
u&=|q|^2, & v&=|p|^2, & \psi&=p(q^\ast)^3,\label{eq:scalar-invariants}\\
\rho&=a^2+d^2+u+v, & \delta&=a^2+u-d^2-v,\nonumber\\
\Theta&=u\bigl((2v+3u)a-3ud\bigr)-(a+d)\operatorname{Re}\psi.&&&&\label{eq:scalar-theta}
\end{align}
Here $q^\ast$ denotes the complex conjugate of $q$, so the convention in \eqref{eq:scalar-invariants}
conjugates $q$ and not $p$. Equivalently, $\psi^*=p^\ast q^3$.
For an explicit expression in real coordinates, write
$q=q_R+iq_I$ and $p=p_R+ip_I$, with all four components real. Then
\begin{align*}
\operatorname{Re}\psi&=p_R(q_R^3-3q_Rq_I^2)+p_I(3q_R^2q_I-q_I^3),\\
\operatorname{Im}\psi&=p_I(q_R^3-3q_Rq_I^2)-p_R(3q_R^2q_I-q_I^3).
\end{align*}
Under the residual circle \eqref{eq:residual-circle}, $q$ and $p$ have weights one and three.  Consequently $u,v,\psi,\rho,\delta,\Theta$ are circle-invariant, with homogeneous degrees $2,2,4,2,2,5$, respectively.

The remaining degree-five invariant can likewise be evaluated entirely in real coordinates:
\[
\begin{aligned}
\Theta={}&(q_R^2+q_I^2)\bigl[(2p_R^2+2p_I^2+3q_R^2+3q_I^2)a
                       -3(q_R^2+q_I^2)d\bigr]\\
 &-(a+d)\bigl[p_R(q_R^3-3q_Rq_I^2)+p_I(3q_R^2q_I-q_I^3)\bigr].
\end{aligned}
\]
\begin{lemma}[Common fixed vectors in scalar coordinates]
\label{lem:scalar-vectors}
The pair $(Z,\Ad_S Z)$ has a nonzero common fixed vector if and only if at least one of the following holds:
\begin{alignat}{2}
&u=0; &&\label{eq:vector-scalar-one}\\
&a=0,\qquad d^2+v=u; &&\label{eq:vector-scalar-two}\\
&d=-a,\qquad v=u; &&\label{eq:vector-scalar-three}\\
&d=a,\qquad up=q^3. &&\label{eq:vector-scalar-four}
\end{alignat}

\end{lemma}

\begin{proof}
When $q=0$, $q_a=ae_1$ and $Ze_1=0$, so $e_1$ is a common fixed vector.  Suppose $q\ne0$.  The residual circle and positive rescaling reduce the coordinates to
\[
q_a=(a,1,0),\qquad q_b=(d,b,c).
\]
By \Cref{prop:failure}, it suffices to look for a kernel vector in one of the two eigenspaces of $S$.

A vector in the positive eigenspace is
\[
s e_1+t(e_2+e_4)+r(e_3+e_5).
\]
Its component in $P_A$ must be parallel to $(a,1,0)$.  A nonzero solution therefore has $(s,t,r)=\lambda(a,1,0)$ with $\lambda\ne0$, and its quaternionic normal coefficient is the real scalar $\lambda$.  The equation $q_by=yq_a$ then gives $d=a$, $b=1$, $c=0$.

A negative eigenvector has the form
\[
s(e_2-e_4)+t(e_3-e_5)+r e_6+z e_7.
\]
If its component in $P_A$ is zero, its quaternionic normal coefficient is $y=r e_2+z e_3\ne0$.  Conjugation $q_b=yq_ay^{-1}$ reverses the $e_1$ coordinate and reflects the transverse unit direction.  As $[r:z]$ varies, all transverse unit directions occur.  This gives exactly $d=-a$ and $b^2+c^2=1$.

If the component in $P_A$ is nonzero, parallelism with $(a,1,0)$ forces $a=0$, $t=0$, and $s\ne0$.  Rescale to $s=1$.  The normal coefficient is $y=-1+r e_2+z e_3$, and $q_b=y e_2y^{-1}$ implies $d^2+b^2+c^2=1$.  Conversely, if this equation holds and $d\ne0$, the explicit vector
\[
e_2-e_4+\frac{c}{d}e_6+\frac{1-b}{d}e_7
\]
is annihilated by $Z$.  If $d=0$, the purely normal case just treated supplies a kernel vector.  Thus the negative-eigenvector possibilities are exactly the two stated families.

Undoing the circle action and rescaling sends the positive condition $b=1,c=0$ to $up=q^3$, and the two negative conditions to \eqref{eq:vector-scalar-two} and \eqref{eq:vector-scalar-three}.  This proves both directions.
\end{proof}

\begin{proof}[Proof of Theorem~C (scalar equivalence)]
By \Cref{lem:scalar-vectors}, conditions (i)--(iii) give a common fixed vector.  For condition (iv), \eqref{eq:terminal-qa-qb} gives $a=d=0$ if and only if $x_2=x_3=0$.  Moreover, $q=-\omega$ and $p=3z$ imply
\[
\operatorname{Im}\psi=\operatorname{Im}\bigl(p(q^\ast)^3\bigr)
 =-3\operatorname{Im}\bigl(z(\omega^\ast)^3\bigr)
 =3\operatorname{Im}(z^\ast\omega^3).
\]
For $q\ne0$, condition (v) is exactly the resonant family \eqref{eq:resonant-qb}: writing $q=q_R+iq_I$ gives $d=3a$ and $p=3(q_R+iq_I)^3/(q_R^2+q_I^2)$.  Its points with $q=0$ already satisfy (i).

Consider (vi) with $q\ne0$.  Normalize again to $q_a=(a,1,0)$ and $q_b=(d,b,c)$.  Then $\operatorname{Im}\psi=c$, so $c=0$.  Put $h=a-d$ and $k=a+d$.  The remaining two equations become
\[
hk=b^2-1,\qquad (b^2+3)h+b(b-1)k=0.
\]
If $b=1$, then $h=0$, yielding the positive common-vector family.  The value $b=0$ is impossible.  Otherwise,
\[
k^2=-\frac{(b+1)(b^2+3)}{b}.
\]
Over the reals, this forces $-1\le b<0$.  At $b=-1$, both $h$ and $k$ vanish, so the generated algebra has a common fixed vector.  On $-1<b<0$, set $\tau=-b$.  These are precisely the real normal forms in \Cref{thm:rankfour-context-complete}; they preserve the quaternionic subalgebras constructed there.  Thus (vi) always implies failure; its $q=0$ points are already in (i).

Conversely, the common-vector locus is exhausted by \Cref{lem:scalar-vectors}.  Its last case, $d=a$ and $up=q^3$, belongs to (vi) when $u>0$: it gives $v=u$, $\operatorname{Im}\psi=0$, $\operatorname{Re}\psi=u^2$, and hence $\delta=\Theta=0$.  It belongs to (i) when $u=0$.  The projective and resonant families have already been identified with (iv) and (v), up to common-vector points.  Every rank-four invariant-subalgebra family satisfies (vi) by its normal form.  The geometric classification \Cref{thm:complete-failure-set} therefore proves exhaustiveness.
\end{proof}

\begin{proposition}[Irredundance of the six scalar conditions]\label{prop:irredundant-tests}
None of the six conditions in \Cref{thm:scalar-classifier} is contained in the union of the other five.
\end{proposition}
\begin{proof}
The following six coordinate vectors in $\mathfrak k_A$ are witnesses.  Direct substitution in \eqref{eq:scalar-invariants}--\eqref{eq:scalar-theta} shows that row $j$ satisfies condition $j$ and no other condition in \eqref{eq:six-tests}.

\begin{center}
\begin{tabular}{cc}\toprule
Condition & $(x_0,x_1,x_2,x_3,x_4,x_5)$\\\midrule
1 & $(0,0,1,0,0,0)$ \\
2 & $(3,0,1,2,0,0)$ \\
3 & $(1,0,1,-1,0,-1)$ \\
4 & $(1,0,0,0,0,0)$ \\
5 & $(0,0,1,1,0,1)$ \\
6 & $(1,0,2,1,0,-1)$ \\
\bottomrule\end{tabular}
\end{center}
Thus omitting any condition would omit a genuine part of the nongenerating set.
\end{proof}

Define six nonnegative real homogeneous polynomials by
\begin{align}
f_1&=u,\nonumber\\
f_2&=a^2\rho+(d^2+v-u)^2,\nonumber\\
f_3&=(a+d)^2\rho+(v-u)^2,\nonumber\\
f_4&=(a^2+d^2)\rho^3+(\operatorname{Im}\psi)^2,\label{eq:classifier-factors}\\
f_5&=(d-3a)^2\rho^2+|up-3q^3|^2,\nonumber\\
f_6&=\delta^2\rho^3+(\operatorname{Im}\psi)^2\rho+\Theta^2,\nonumber
\end{align}
and put
\begin{equation}
\mathcal P(Z)=\prod_{j=1}^{6}f_j(Z).
\label{eq:exact-polynomial}
\end{equation}

\begin{corollary}[Homogeneous encoding of the six scalar tests]
\label{thm:polynomial-classifier}
$\mathcal P$ is circle-invariant and homogeneous of degree $34$, and
\[
\mathcal N=\{\mathcal P=0\},\qquad
\mathcal G=\{\mathcal P>0\}.
\]
In particular, the complete nongenerating set is real algebraic, and its complement is a nonempty real Zariski-open set.  
\end{corollary}

\begin{proof}
For $Z\ne0$ one has $\rho>0$.  Consequently a factor in
\eqref{eq:classifier-factors} vanishes if and only if all its nonnegative summands vanish.
In order, this gives exactly conditions (i)--(vi) of \Cref{thm:scalar-classifier}.
Thus $\mathcal P(Z)=0$ if and only if at least one of the six failure tests holds.
At the origin every factor vanishes and the generated algebra is zero, so the same
criterion remains valid there.

The factor degrees are $(2,4,4,8,6,10)$, whose sum is $34$.
The quantities $u,v,\psi,\rho,\delta,\Theta$ are invariant under the residual circle;
$up-3q^3$ has weight three, so its squared modulus is invariant as well.
This proves homogeneity, nonnegativity, and invariance of the product.
Since $S(H_A)=H_B\ne H_A$, Theorem~A provides a generating element of $\mathfrak k_A$.
The scalar equivalence already proved then gives $\mathcal P>0$ at that element.
Hence the zero set of $\mathcal P$ is proper and its complement is nonempty and real Zariski open.
\end{proof}

The powers of $\rho$ homogenize sums of squares without changing their zero sets.

\begin{corollary}[Comparison with the uniform determinant]\label{cor:fixed-R30}
For the fixed pair $(\mathfrak k_A,S)$,
\begin{equation}
\mathcal N=\{Z\in\mathfrak k_A:\mathcal R(S,Z)=0\}
=\{Z\in\mathfrak k_A:\mathcal P(Z)=0\}.\label{eq:R30-P34}
\end{equation}
\end{corollary}
\begin{proof}
\Cref{cor:no-principal} excludes the irreducible proper outcome for every element of the prescribed stabilizer with this $S$.
The factorization theorem \Cref{thm:universal-locus} therefore reduces generation to $\mathcal R(S,Z)>0$.
The scalar encoding in \Cref{thm:polynomial-classifier} has the same zero set.
\end{proof}
The determinant and the six systems describe the same real set: the former provides one equation, while the latter display its geometric families.

\begin{corollary}[Codimension and connectivity]
\label{cor:dimension-connectivity}
The real algebraic set $\mathcal N\subset\mathfrak k_A\cong\RR^6$ has dimension four.  Its complement $\mathcal G$ has full Lebesgue measure, and any two elements of $\mathcal G$ can be joined by two line segments lying entirely in $\mathcal G$.  In particular, $\mathcal G$ and its intersection with the unit sphere are path connected.
\end{corollary}

\begin{proof}
The common-vector set in \Cref{lem:scalar-vectors} has dimension at most four.  Indeed, its first family is a four-plane.  Its second and third families are respectively defined, after one linear restriction, by a nonzero quadratic equation in five real variables.  In the fourth, $q\ne0$ determines $p=q^3/|q|^2$, leaving the three parameters $(a,q)$; its $q=0$ part is also at most three-dimensional.  The four-plane $u=0$ gives the matching lower bound.

The projective family is a quartic hypersurface in a four-plane and has dimension three by \Cref{cor:quartic-cone}.  The resonant and rank-four families have dimension at most three by their explicit parametrizations.  \Cref{thm:complete-failure-set} therefore gives $\dim_\RR\mathcal N=4$.  These dimension statements use the usual dimension of semialgebraic sets; the parameter descriptions also provide finite smooth stratifications, so a set of dimension at most five in $\RR^6$ has Lebesgue measure zero.

The remaining path-connectivity assertion is a general consequence for complements of closed semialgebraic sets of codimension at least two, applied here to the dimension just computed.  Fix $z\in\mathcal G$.  The intermediate endpoints $w$ for which the segment $[z,w]$ meets $\mathcal N$ are contained in the semialgebraic image
\[
B_z=\{z+t(n-z):n\in\mathcal N,\ t\ge1\}.
\]
This set has dimension at most five: the dimension of a semialgebraic image does not exceed that of its parameter space \cite[\S2.8]{BochnakCosteRoy1998}, here $\dim\mathcal N+1=5$.  Given $z_1,z_2\in\mathcal G$, choose $w\notin B_{z_1}\cup B_{z_2}$.  The two segments $[z_1,w]$ and $[w,z_2]$ avoid $\mathcal N$, as required.  The unit-sphere assertion follows by radial projection of this path: the origin lies in $\mathcal N$, and $\mathcal G$ is invariant under nonzero rescaling by homogeneity of $\mathcal P$.
\end{proof}

\appendix
\section{Block calculations for the Gram certificates}\label{app:gram-certificates}
The two determinant identities used in \Cref{thm:visible-contexts,thm:rankfour-context-complete} can be checked through small blocks rather than a full $6\times6$ expansion.  We use only the trace inner product, parity under $\Ad_S$, and the explicit derivation matrix in the proof of \Cref{thm:low-complete}.  In that matrix the specializations below always have $c=0$.

\subsection{Why three blocks occur}\label{app:paired-blocks}
For an even $A$ and an odd $B$, set
\[
C=[A,B],\qquad C_+=[A,C],\qquad C_-=[B,C],\qquad D=[A,C_+].
\]
The words $A,C_-$ are even and $B,C_+,C,D$ are odd.  Even and odd words are orthogonal.  Within the odd part, trace invariance gives
\[
\langle B,C\rangle=\langle C_+,C\rangle=\langle C_+,D\rangle=0,
\qquad \langle B,D\rangle=-\langle C,C_+\rangle=0.
\]
Consequently the Gram matrix in the order
\[
\mathcal E=(A,C_-,\ B,C_+,\ C,D)
\]
is block diagonal.  Introduce six squared norms
\[
\mathsf a=\|A\|^2,\quad\mathsf b=\|B\|^2,\quad\mathsf c=\|C\|^2,
\quad\mathsf p=\|C_+\|^2,\quad\mathsf m=\|C_-\|^2,\quad\mathsf d=\|D\|^2.
\]
The three blocks are exactly
\[
G_0=\begin{pmatrix}\mathsf a&\mathsf c\\\mathsf c&\mathsf m\end{pmatrix},\qquad
G_1=\begin{pmatrix}\mathsf b&-\mathsf c\\-\mathsf c&\mathsf p\end{pmatrix},\qquad
G_2=\begin{pmatrix}\mathsf c&-\mathsf p\\-\mathsf p&\mathsf d\end{pmatrix}.
\]
For example, $\langle A,[B,C]\rangle=\|C\|^2$ and $\langle C,[A,C_+]\rangle=-\|C_+\|^2$, explaining the off-diagonal entries and their signs.  Thus
\[
\det\operatorname{Gram}(\mathcal E)
=(\mathsf a\mathsf m-\mathsf c^2)
 (\mathsf b\mathsf p-\mathsf c^2)
 (\mathsf c\mathsf d-\mathsf p^2).
\]
These words are successive brackets of $A,B$: they test whether the generated algebra has acquired six independent directions.  Once invariance has placed all six in a quaternionic stabilizer, a nonzero determinant proves equality with that six-dimensional stabilizer.  No classification of its smaller subalgebras is needed for this step.

\subsection{The resonant certificate}\label{app:resonant-blocks}
First set $(a,c,s)=(t,1,0)$, with $t\in\RR$.  The normalized resonant element is
\[
U_t=-D_{13}+t(D_{23}+D_{45})+D_{46},\qquad
q_a=(t,1,0),\quad q_b=(3t,3,0).
\]
In the matrix from the proof of \Cref{thm:low-complete}, substitute $(a,d,b,c)=(t,3t,3,0)$.  Form $A=(U_t+SU_tS)/2$, $B=(U_t-SU_tS)/2$ and the brackets above.  Write $\sigma=2t^2+1$, $\lambda=t^2+1$.  Taking the six squared trace norms gives the following three blocks:
\[
G_0=24\sigma\begin{pmatrix}1&14\\14&244\end{pmatrix},\qquad
G_1=24\begin{pmatrix}1&-14\sigma\\-14\sigma&244\sigma^2\end{pmatrix},\qquad
G_2=48\sigma\begin{pmatrix}7&-122\sigma\\-122\sigma&2188\sigma^2\end{pmatrix}.
\]
Their determinants are, respectively,
\[
2^{10}3^3\sigma^2,\qquad 2^{10}3^3\sigma^2,\qquad 2^{12}3^5\sigma^4.
\]
Indeed, the only constant subtractions needed are $244-14^2=48$ and $7\cdot2188-122^2=432$.  The matrix calculation also gives
\[
F_1=A+B,\quad F_2=A-B,\quad F_3=-2C,\quad
F_4=-2(C_++C_-),\quad F_5=-2(C_+-C_-),
\]
\[
F_6=-\frac{4\lambda}{\sigma}D.
\]
The last identity can be checked by multiplying the same specialized matrices; it is an identity for this family, not a general commutator identity.  The absolute determinant of this change from $\mathcal E$ to $(F_1,\ldots,F_6)$ is
$2\cdot2\cdot8\cdot4\lambda/\sigma=2^7\lambda/\sigma$.
Since a Gram determinant changes by the square of the basis determinant, we obtain
\[
\det\operatorname{Gram}(F_1(U_t),\ldots,F_6(U_t))
=2^{46}3^{11}(t^2+1)^2(2t^2+1)^6.
\]

To restore the parameters, put $r=\sqrt{c^2+s^2}>0$ and $c+is=re^{i\theta}$.  The residual circle action \eqref{eq:residual-circle}, which commutes with $\Ad_S$, gives
\[
\widehat Z_{a;c,s}=r^3\Ad_{g_\theta}U_{a/r}.
\]
This is checked directly in the quaternionic coordinates: the first pair is $r^2(a,c,s)$ and the second has axial coordinate $3ar^2$ and transverse coordinate $3(c+is)^3$.  The six words have lengths $(1,1,2,3,3,4)$, whose sum is $14$.  Scaling the generator by $r^3$ therefore scales their Gram determinant by $r^{84}$.  Substituting $t=a/r$ and $n=r^2$ yields
\[
r^{84}(t^2+1)^2(2t^2+1)^6
=n^{34}(a^2+n)^2(2a^2+n)^6,
\]
which proves \eqref{eq:resonant-so4-gram}, including $a=0$.

\subsection{The rank-four certificate}\label{app:rankfour-blocks}
Use the normalized element $\widetilde Z=Z_{\tau,\kappa}/(6\nu_\tau)$ and take its even and odd parts as $A,B$.  In the same explicit matrix substitute
\[
a=\frac{3-\tau}{2\nu_\tau}\kappa,\qquad
d=\frac{2\tau^2+\tau+3}{2\nu_\tau}\kappa,\qquad b=-\tau,\qquad c=0,
\]
and reduce powers of $\kappa$ using $\kappa^2=(1-\tau)\nu_\tau/\tau$.
For compact formulas write
\[
\nu=\tau^2+3,\qquad \eta=3-\tau,\qquad
\chi=\tau^2-2\tau+3,\qquad L=\tau^2-4\tau+7.
\]
Here $L$ is a scalar polynomial, not the generated algebra $L_Z$.  The six squared trace norms needed in Appendix~\ref{app:paired-blocks} are
\begin{align*}
\mathsf a&=\frac{\eta\nu}{\tau},&
\mathsf b&=\frac{(1+\tau)(\chi^2+12\tau^2)}{\tau\nu},&
\mathsf c&=\eta(1+\tau)L,\\
\mathsf p&=\frac{\eta^2(1+\tau)(\tau^2\chi^2+12)}{\tau\nu},&&&\\[4pt]
\mathsf m&=\frac{\eta(1+\tau)^2\bigl(\tau^2L^2+3\eta^2(1-\tau)^2\bigr)}{\tau\nu},&&&\\[4pt]
\mathsf d&=\frac{\eta^3(1+\tau)(\tau^4\chi^2+12)}{\tau^2\nu}.&&&
\end{align*}
Thus the full intermediate Gram matrix is specified by six scalar entries and the three printed $2\times2$ blocks; no unprinted matrix is required.  The reductions in their determinants are particularly short because $\nu L=\chi^2+12$.  In addition to canceling the $\tau^2L^2$ terms in $\mathsf a\mathsf m-\mathsf c^2$, use
\begin{align*}
(\chi^2+12\tau^2)(\tau^2\chi^2+12)-\tau^2(\chi^2+12)^2
 &=12(1-\tau^2)^2\chi^2,\\
(\chi^2+12)(\tau^4\chi^2+12)-(\tau^2\chi^2+12)^2
 &=12(1-\tau^2)^2\chi^2.
\end{align*}
The three block determinants are therefore
\begin{align*}
\det G_0&=\frac{3\eta^4(1-\tau)^2(1+\tau)^2}{\tau^2},\\
\det G_1&=\frac{12\eta^2(1-\tau)^2(1+\tau)^4\chi^2}{\tau^2\nu^2},\\
\det G_2&=\frac{12\eta^4(1-\tau)^2(1+\tau)^4\chi^2}{\tau^2\nu^2}.
\end{align*}
Every factor is positive for $0<\tau<1$.  Reordering the six words in the proof of \Cref{thm:rankfour-context-complete} to $\mathcal E$ does not change the Gram determinant.  Restoring $Z_{\tau,\kappa}=6\nu\widetilde Z$ contributes $(6\nu)^{28}$, because the word lengths again sum to $14$.  Since $3\cdot12\cdot12=2^4 3^3$, multiplication gives exactly
\[
\frac{2^{32}3^{31}(3-\tau)^{10}(1-\tau)^6(1+\tau)^{10}
\nu^{24}(\tau^2-2\tau+3)^4}{\tau^6},
\]
the expression in \eqref{eq:rankfour-new-gram}.  This also identifies the source of the denominator $\tau^6$ and proves positivity without numerical sampling.

\paragraph{Verification of finite calculations.}\label{sec:reproducibility}
All finite calculations used in the proofs are specified by the displayed matrices, polynomial identities, and determinant formulas in the article.  The symbolic arithmetic was additionally checked with SymPy using exact rational and algebraic expressions; no numerical tolerance is used in any proof.

\paragraph{Use of AI-assisted tools.}
OpenAI ChatGPT was used for drafting and revising the exposition, organizing the manuscript, and assisting with symbolic calculations.  Anthropic Claude was used for an additional audit of arguments and symbolic calculations in the preceding manuscript.  The authors retain responsibility for the mathematical statements, proofs, calculations, citations, and final text.  The present expository revision and its added calculations are subject to their review before submission.


\begin{thebibliography}{99}\small
\bibitem{JurdjevicSussmann1972} V. Jurdjevic and H. J. Sussmann, \textit{Control Systems on Lie Groups}, J. Differential Equations \textbf{12}, 313--329 (1972).
\bibitem{BauerLevaillantFreedman2014} B. Bauer, C. Levaillant, and M. Freedman, \textit{Universality of Single Quantum Gates}, arXiv:1404.7822 (2014).
\bibitem{DraperMartin2025} C. Draper and C. Mart\'in-Gonz\'alez, \textit{A Perspective on Totally Geodesic Submanifolds of the Symmetric Space $G_2/\SO(4)$}, Trans. Amer. Math. Soc. \textbf{378}(12), 8689--8721 (2025); doi:10.1090/tran/9479; arXiv:2504.07586.
\bibitem{Kuranishi1951} M. Kuranishi, \textit{On Everywhere Dense Imbedding of Free Groups in Lie Groups}, Nagoya Math. J. \textbf{2}, 63--71 (1951).
\bibitem{AlbuquerqueSilvaLeite1989} H. Albuquerque and F. Silva Leite, \textit{On the Generators of Semisimple Lie Algebras}, Linear Algebra Appl. \textbf{119}, 51--56 (1989).
\bibitem{Bois2009} J.-M. Bois, \textit{Generators of Simple Lie Algebras in Arbitrary Characteristics}, Math. Z. \textbf{262}, 715--741 (2009); arXiv:0708.1711.
\bibitem{DetinkoDeGraaf2020} A. S. Detinko and W. A. de Graaf, \textit{2-Generation of Simple Lie Algebras and Free Dense Subgroups of Algebraic Groups}, J. Algebra \textbf{545}, 159--173 (2020); doi:10.1016/j.jalgebra.2019.06.012.
\bibitem{Chirvasitu2021} A. Chirvasitu, \textit{Large Sets of Generating Tuples for Lie Groups}, arXiv:2106.11955 (2021).
\bibitem{HarveyLawson1982} R. Harvey and H. B. Lawson, Jr., \textit{Calibrated Geometries}, Acta Math. \textbf{148}, 47--157 (1982).
\bibitem{Cacciatori2005} S. L. Cacciatori, B. L. Cerchiai, A. Della Vedova, G. Ortenzi, and A. Scotti, \textit{Euler Angles for $G_2$}, J. Math. Phys. \textbf{46}, 083512 (2005); arXiv:hep-th/0503106.
\bibitem{KnarrStroppel2025} N. Knarr and M. J. Stroppel, \textit{Subalgebras of Octonion Algebras}, J. Algebra \textbf{664}, 42--74 (2025); doi:10.1016/j.jalgebra.2024.10.004; arXiv:2303.00335.
\bibitem{ChemtovKarigiannis2022} M. Chemtov and S. Karigiannis, \textit{Observations about the Lie Algebra $\mathfrak g_2\subset\mathfrak{so}(7)$, Associative $3$-Planes, and $\mathfrak{so}(4)$ Subalgebras}, Expo. Math. \textbf{40}, 845--869 (2022); arXiv:2209.10613.
\bibitem{Baez2002} J. C. Baez, \textit{The Octonions}, Bull. Amer. Math. Soc. \textbf{39}, 145--205 (2002); arXiv:math/0105155.
\bibitem{SpringerVeldkamp2000} T. A. Springer and F. D. Veldkamp, \textit{Octonions, Jordan Algebras and Exceptional Groups}, Springer, Berlin (2000).
\bibitem{RauschSlupinski2022} M. Rausch de Traubenberg and M. J. Slupinski, \textit{Commutation Relations of $\mathfrak g_2$ and the Incidence Geometry of the Fano Plane}, arXiv:2207.13946 (2022).
\bibitem{Dynkin1952b} E. B. Dynkin, \textit{Semisimple Subalgebras of Semisimple Lie Algebras}, Mat. Sbornik N.S. \textbf{30(72)}, 349--462 (1952); English transl., Amer. Math. Soc. Transl. Ser. 2 \textbf{6}, 111--244 (1957).
\bibitem{SilvaLeiteCrouch1988} F. Silva Leite and P. E. Crouch, \textit{Controllability on Classical Lie Groups}, Math. Control Signals Systems \textbf{1}, 31--42 (1988).
\bibitem{AgrachevSachkov2004} A. A. Agrachev and Y. L. Sachkov, \textit{Control Theory from the Geometric Viewpoint}, Springer, Berlin (2004).
\bibitem{ZeierSchulteHerbrueggen2011} R. Zeier and T. Schulte-Herbr\"uggen, \textit{Symmetry Principles in Quantum Systems Theory}, J. Math. Phys. \textbf{52}, 113510 (2011); doi:10.1063/1.3657939; arXiv:1012.5256.
\bibitem{BochnakCosteRoy1998} J. Bochnak, M. Coste, and M.-F. Roy, \textit{Real Algebraic Geometry}, Ergebnisse der Mathematik und ihrer Grenzgebiete (3), vol. 36, Springer, Berlin (1998); doi:10.1007/978-3-662-03718-8.
\end{thebibliography}
\end{document}